\documentclass[10pt]{article}
\usepackage[utf8]{inputenc}
\usepackage[a4paper,left=3cm, right= 3cm, top=2cm, bottom= 2cm]{geometry}
\usepackage{amsmath}
\usepackage{amsfonts}
\usepackage{amssymb}
\usepackage{amsthm}
\usepackage{bbold}
\usepackage{authblk}

\usepackage[active]{srcltx}
\usepackage{ifpdf}
\ifpdf 
    \usepackage[pdftex,     
            plainpages=false,   
            breaklinks=true,    
            colorlinks=true,
            pdftitle=My Document
            pdfauthor=My Good Self
           ]{hyperref} 
\else 
    \usepackage{hyperref}       
\fi

\DeclareMathOperator{\dv}{div}
\newcommand{\tor}{\mathbb{T}^d}
\newcommand{\T}{\mathbb{T}^d}
\newcommand{\jQsum}{\sum_{j=1}^{2d}\sum_{Q\in\mathcal{Q}}}
\newcommand{\W}{W^{Q,j}}
\DeclareMathOperator{\sign}{sign}
\DeclareMathOperator{\supp}{supp}

\theoremstyle{plain}
\newtheorem{thm}{Theorem}[section]
\newtheorem{lemma}[thm]{Lemma}
\newtheorem{prop}[thm]{Proposition}
\newtheorem{conj}[thm]{Conjecture}

\theoremstyle{definition}

\newtheorem{rem}[thm]{Remark}

\usepackage{color}

\def\R{\mathbb{R}}

\def\N{\mathbb{N}}
\def\Z{\mathbb{Z}}
\def\Lip{\mathrm{Lip}}
\def\T{\mathbb{T}}

\def\div{\operatorname{div}}

\def\supp{\mathrm{supp\, }}
\def\e{\varepsilon}
\def\sign{\mathrm{sign}}

\title{On the failure of the chain rule \\ for the divergence of Sobolev vector fields}

\author[1]{Miriam Buck}
\affil[1]{Technische Universit\"at Darmstadt, Fachbereich Mathematik, D-64285 Darmstadt, Germany}
\author[1]{Stefano Modena}

\date{}                     
\begin{document}
\maketitle

\abstract{We construct  a large class of incompressible vector fields with Sobolev regularity, in dimension $d \geq 3$, for which the \emph{chain rule problem} has a negative answer. In particular, 
for any renormalization map $\beta$ (satisfying suitable assumptions) and any (distributional) renormalization defect  $T$ of the form $T = \div h$, where $h$ is an $L^1$ vector field, we can construct an incompressible  Sobolev vector field $u \in W^{1, \tilde p}$ and a density $\rho \in L^p$ for which $\div (\rho u) =0$ but $\div (\beta(\rho) u) = T$, provided $1/p + 1/\tilde p \geq 1 + 1/(d-1)$}.

\section{Introduction}

This note deals with the problem of the \emph{chain rule} for the divergence of vector fields, which can be stated as follows (see {\cite{ambrosio2007chain, crippa2014failure}}):

\smallskip

\textbf{Q1}: \noindent \emph{Given a periodic vector field $u : \T^d \to \R^d$ and a periodic scalar function $\rho: \T^d \to \R$, assuming that $\div u = 0$ and $\div (\rho u) = 0$ in the sense of distributions, can one deduce that also $\div (\beta(\rho) u) = 0$ in the sense of distributions as well, for every given smooth $\beta: \R \to \R$?} 

\smallskip

or, in a more general setting, 

\smallskip

\textbf{Q2}: \emph{Given $u, \rho$ as before, is it possible to express the quantity $\div (\beta(\rho) u)$, where $\beta : \R \to \R$ is a smooth given function, only in terms of the quantities $\div u$ and $\div (\rho u)$?}

\smallskip

Here $\T^d = \R^d/\Z^d$ is the $d$-dimensional flat torus. Similar questions can be asked for more general domains and boundary conditions, but we prefer to stick to the period case for the sake of simplicity. Clearly, if $\rho, u$ are smooth, it holds
\begin{equation}
\label{eq:chain-rule}
\div (\beta(\rho) u) = \beta'(\rho) \div (\rho u) + \left[ \beta(\rho) - \rho \beta'(\rho) \right] \div u, 
\end{equation}
and thus the answer to both questions above is positive. 

If, on the contrary, $\rho, u$ are not smooth, the answer is far from trivial, and it is intimately connected to the problem of uniqueness of solutions to the Cauchy problem associated to the linear transport equation
\begin{equation}
\label{eq:transport}
\partial_t \rho + \nabla \rho \cdot u = 0.
\end{equation}
Given a vector field $u$, a distributional solution $\rho$ to \eqref{eq:transport} is called \emph{renormalized} if $\beta(\rho)$ is still a distributional solution to \eqref{eq:transport} for every $C^1$ map $\beta$, with bounded derivative (such $\beta$'s are called \emph{renormalization maps}). The vector field $u$ has the \emph{renormalization property} if every solution $\rho$ to \eqref{eq:transport} (in a given class of densities) is renormalized. It is a well known fact dating back to the 1989 work of DiPerna and Lions {\cite{diperna1989ordinary}} that, if $u$ has the renormalization property, then solutions to the Cauchy problem associated to \eqref{eq:transport} are unique (in the corresponding class of densities). 

The chain rule question Q1 is then a sort of ``stationary version'' of the renormalization property and it is thus natural to expect that vector fields for which uniqueness of solutions to the Cauchy problem associated to \eqref{eq:transport} holds are exactly those vector fields for which the chain rule (in the sense of Q1) holds.

Indeed, DiPerna and Lions in the mentioned paper {\cite{diperna1989ordinary}} proved uniqueness of solutions for the Cauchy problem associated to \eqref{eq:transport} in the case of (time dependent) Sobolev vector  fields $u \in L^1 ([0,T]; W^{1,p'} (\T^d))$ with bounded divergence in the class of densities $\rho \in L^\infty ([0,T], L^p(\T^d))$, for every given $p \in [1,\infty]$, and the same proof can be used to show that, if $u \in W^{1,p'}(\T^d)$ and $\rho \in L^p(\T^d)$, then the answer to Q1 (and in some sense also to Q2) is positive. DiPerna and Lions' uniqueness result for the transport equation \eqref{eq:transport} was then extended by Ambrosio in 2004 to the case of incompressible vector fields $u \in L^1 ([0,T], BV(\T^d))$  with bounded divergence, thus similarly yielding a positive answer to Q1 and Q2 for stationary vector fields $u \in BV(\T^d)$ with bounded divergence. In both results crucial ingredients are the bound on one full derivative of the vector field ($\nabla u \in L^{p'}$ in the case of DiPerna and Lions, $\nabla u$ in the space of measures in Ambrosio's case) and the control on $\div u$, i.e., roughly speaking, on the rate of compression of the ``flow'' generated by $u$. 

Ambrosio's result does not cover the case of $BV$ vector fields whose divergence (which, in general, is a Radon measure) admits also jump and/or Cantor parts. This case was then analyzed in 2007 by Ambrosio, De Lellis and Malý in {\cite{ambrosio2007chain}}, where they gave a detailed answer to Q1 and  Q2 for \emph{nearly incompressible} vector fields $u \in BV(\T^d)$ whose divergence has no Cantor part, and a partial answer to the case when $\div u$ has nonzero Cantor part. The notion of \emph{nearly incompressibility} is a natural replacement of the bounded divergence condition in DiPerna-Lions' and Ambrosio's theory, and ensures that the product $\rho \div u$ (see \eqref{eq:chain-rule}) between the almost everywhere defined map $\rho$ and the possibly singular measure $\div u$ 
is well defined in the sense of of distributions. Here we do not want to enter into more details and we refer for a definition to {\cite[Def. 7]{ambrosio2007chain}}. More recently (2016) the case of two dimensional  \emph{nearly incompressible} $BV$ vector fields (with divergence possibly having also Cantor part) was accurately discussed by Bianchini and Gusev in \cite{bianchini2016steady}, building on the description of level sets for Lipschitz continuous functions developed in {\cite{alberti2013structure}}. A comprehensive analysis of the chain rule for the divergence of $BV$ nearly incompressible vector fields in $\R^d$, for any $d \geq 2$, was finally performed by Bianchini and Bonicatto in 2020 \cite{bianchini2020uniqueness}.

In these notes we are interested in negative answers to questions Q1 and Q2. More precisely, we ask whether, for any \emph{given} renormalization function $\beta$ and any \emph{given} distribution $T$ (usually called \emph{renormalization defect}), it is possible to construct a density $\rho$ and a vector field $u$ (with the best possible integrability and regularity) such that 
\begin{equation}
\label{eq:div-u-rhou}
\div u = \div (\rho u) = 0,
\end{equation}
but 
\begin{equation}
\label{eq:div-defect}
\div (\beta(\rho) u)  = T
\end{equation}
in the sense of distributions. Notice that this problem is more general than just asking whether there are $\rho, u$ for which question Q1 admits a negative answer. Indeed, a negative answer to Q1 would just correspond to finding $\rho,u$ for which \eqref{eq:div-u-rhou} holds, but the l.h.s. in \eqref{eq:div-defect} is nonzero for some $\beta$, and this could be  done  by adapting example of nonuniqueness for the transport equation \eqref{eq:transport} (see e.g. \cite{tsuruhashi2021}). On the contrary, we want to \emph{prescribe} a priori the renormalization function $\beta$ and, more importantly, the renormalization defect $T$ in \eqref{eq:div-defect}, and to find $\rho,u$ for which the l.h.s. in \eqref{eq:div-defect} is exactly given by $T$. 

To the best of our knowledge the only result in this direction is contained in {\cite{crippa2014failure}} where the authors use a convex integration approach based on the theory of laminates to show that for every strongly convex $\beta$ (see {\cite[Sec.2]{crippa2014failure}} for the definition of strongly convex map) and for every distribution $T$ for which the equation $\div w = T$ admits a bounded \emph{continuous} solution, there are \emph{bounded} $\rho$ and $u$ such that \eqref{eq:div-u-rhou} and \eqref{eq:div-defect} hold simultaneously. Notice that here no bound on the derivative of $u$ is available, in particular $u$ is neither Sobolev nor $BV$, but it is merely $L^\infty$. On the other side, the constructed densities $\rho$ are not just in \emph{some} $L^p_{loc}$ space, but they are bounded. 

There has recently been progresses concerning non-uniqueness results for the transport equation \eqref{eq:transport} in the case of Sobolev vector fields, and the state of the art can be roughly summarized in the following theorem (see also {\cite{cheskidov2021nonuniqueness}}, {\cite{giri2021non}, \cite{pitcho2021almost}} for similar statements).

\begin{thm}[Non-uniqueness for the transport equation, {\cite{modena2018non}, \cite{modena2019non}, \cite{modena2020convex}, \cite{brue2021positive}}]
\label{thm:transport}
Let $p,\tilde p \in [1, \infty)$. Assume that 
\begin{equation}
\label{eq:1-over-d}
\frac{1}{p} + \frac{1}{\tilde p} > 1 + \frac{1}{d}.
\end{equation}
Then there is a density $\rho \in C([0,1], L^p(\T^d))$ and an incompressible vector field $u \in C([0,1], L^{p'}(\T^d)) \cap C([0,1], W^{1, \tilde p}(\T^d))$ such that \eqref{eq:transport} holds and $\rho|_{t=0} \equiv 1$ but $\rho|_{t = 1} \not \equiv 1$. Moreover, if $p>1$, the density $\rho$  is strictly positive. 
\end{thm}
It is expected that the last sentence (positive density) holds also in the case $p=1$, though a proof of this fact is not available and it is probably far from trivial. Motivated by Theorem \ref{thm:transport} and by the strong analogy between the transport equation and the chain rule problem, it is natural to conjecture the following. 

\begin{conj}
\label{conj:main}
Let $p,\tilde p \in [1, \infty)$ and assume that \eqref{eq:1-over-d} holds. Let $\beta : \R \to \R$ be given and $T \in \mathcal{D}'(\T^d)$ be a given distribution (both satisfying some suitable assumption, see Points \eqref{pt:assu-beta}-\eqref{pt:assu-t} in Remark \ref{rem:comments-statement} below). Then there is a density $\rho \in L^p(\T^p)$ and an incompressible vector field $u \in L^{p'}(\T^d) \cap W^{1, \tilde p}(\T^d)$ such that 
\eqref{eq:div-u-rhou} and \eqref{eq:div-defect} hold simultaneously. Moreover, (in certain cases) the density can be chosen to be strictly positive. 
\end{conj}

These notes are a first contribution towards the proof of Conjecture  \ref{conj:main}. We indeed prove the following theorem. 

\begin{thm}\label{thm:main}
Let $d\geq 3$ and let $\beta:\mathbb{R}\rightarrow\mathbb{R}$ be a smooth, globally Lipschitz continuous function with the property that there exist constants $C,D>0$ such that
\begin{equation}\label{grwothcondition}
D|\tau|\leq \beta(\tau)\leq C|\tau| \text{ for } |\tau|\gg 1.
\end{equation}
Let $T \in \mathcal{D}'(\T^d)$ be such that $T=\dv h$ for some $h\in L^1(\tor; \R^d).$ Let $p\in (1,\infty),\tilde{p}\in\lbrack 1,\infty)$ such that
\begin{equation}
\label{eq:d-1}
\frac{1}{p}+ \frac{1}{\tilde{p}}>1+\frac{1}{d-1}.
\end{equation}
Then there exist $\rho:\tor\rightarrow\mathbb{R}, u:\tor\rightarrow\mathbb{R}^d$ with $\rho\in L^p(\tor),u\in W^{1,\tilde{p}}(\tor)\cap L^{p'}(\tor)$ such that \eqref{eq:div-u-rhou} and \eqref{eq:div-defect} hold. 
\end{thm}

\begin{rem}
\label{rem:comments-statement}
We add some comments about the assumptions in Theorem \ref{thm:main} and their differences w.r.t. Conjecture \ref{conj:main}.

\begin{enumerate}
\item \label{pt:assu-beta} \emph{Some} assumption on the renormalization function $\beta$ is needed. Indeed, if $\beta$ is linear (or affine) the statement can not be true. In the paper {\cite{crippa2014failure}} $\beta$ was supposed to be strongly convex. Here, on the contrary, we assume \eqref{grwothcondition}. The upper bound $|\beta(\tau)| \leq C |\tau|$ (for large $|\tau|$) on the linear growth at infinity is quite natural, as it  ensures that  
\begin{equation*}
\rho \in L^p, \ u \in L^{p'} \ \Longrightarrow \ \beta(\rho) u \in L^1
\end{equation*}
and thus $\div (\beta(\rho) u)$ is a well defined distribution.  On the contrary, the lower bound $|\beta(\tau)| \geq D |\tau|$ is needed for our proof  to work, but it would be interesting to understand if it can be removed. The global Lipschitz bound on $\beta$ is reminiscent of the original definition of renormalization function by DiPerna and Lions, where $\beta'$ was assumed to be bounded. We believe this assumption can be removed though we prefer to keep it to avoid further technicalities.

\item  \label{pt:assu-t} \emph{Some} assumption on the defect distribution $T$ is needed as well. Since we construct $\rho, u$ such that $\beta(\rho) u \in L^1$ and  \eqref{eq:div-defect} holds, our assumption that $T$ is the divergence of an $L^1$ function is necessary. Theorem \ref{thm:main} shows that it is also sufficient. 

\item The condition $d \geq 3$ is related to \eqref{eq:d-1}, because \eqref{eq:d-1} is an empty condition if $d =1,2$. 

\item Condition \eqref{eq:d-1} is clearly more restrictive than \eqref{eq:1-over-d} which appears in Conjecture  \ref{conj:main} and  in Theorem \ref{thm:transport}. This is due to the fact that the chain rule problem is stationary, whereas  the proof of Theorem \ref{conj:main} for exponents $p,\tilde p$ in the range 
\begin{equation*}
1 + \frac{1}{d} < \frac{1}{p} + \frac{1}{\tilde p} \leq 1 + \frac{1}{d-1}
\end{equation*}
strongly uses time as a sort of additional spatial dimension. A similar problem arises when one looks for ``anomalous'' stationary solutions to the Navier Stokes equation in dimension $d \geq 3$: though the celebrated paper by Buckmaster and Vicol {\cite{buckmaster2019nonuniqueness}} on (evolutionary) Navier-Stokes equation holds for any $d \geq 3$, the proof for the stationary counterpart works only if $d \geq 4$ (see \cite{luo2019stationary}). Similar problems (in different settings) arise also in \cite{ozanski2021} and \cite{giri2021non}.

\item Similarly, we can prove Theorem \ref{thm:main} for $p>1$, but we are not able to extend the proof to $p=1$, though the analog Theorem \ref{thm:transport} for the transport equation holds also in the case $p=1$ (at least in the case of sign changing densities). This is again due to the fact that the proof of Theorem \ref{thm:transport} in the case $p=1$ again heavily exploits the space-time structure of equation \eqref{eq:transport}. 

\item Finally, we are not able to construct positive densities. This would indeed also not be possible under our assumptions on $\beta$. Indeed, according to Theorem \ref{thm:main}, $\beta(\tau) = |\tau|$ is admissible as renormalization function. With such $\beta$, the second equation in \eqref{eq:div-u-rhou} and \eqref{eq:div-defect} become equivalent if $\rho \geq 0$. From a more technical point of view, we need, for the density, building blocks which change sign (see Remark \ref{rem:sign-changing-density} below). 
\end{enumerate}
\end{rem}

Inspired by the construction in {\cite{giri2021non}}, we can prove that, in even dimension, the vector fields in Theorem \ref{thm:main} can be chosen to be Hamiltonian. This is the precise statement.

\begin{thm}\label{thm:hamiltonian}
Assume $d = 2d'$ be even, $d \geq 4$. Then the same statement as in Theorem \ref{thm:main} holds and, in addition, the vector field $u$ is Hamiltonian, i.e. $u = J\nabla H$ for some function $H:\tor\rightarrow\mathbb{R}\in W^{1,p'}(\tor)\cap W^{2,\tilde{p}}(\tor)$ with
\begin{equation*}
J = \begin{pmatrix}
0_{d'} & I_{d'}\\
-I_{d'} & 0_{d'}
\end{pmatrix}.
\end{equation*}
\end{thm}

The main part of the paper (Sections  \ref{sec:main-prop}-\ref{sec:proofofall}) is devoted to prove Theorem \ref{thm:main}, whereas we will sketch the proof of Theorem \ref{thm:hamiltonian} in Section \ref{sec:hamiltonian}.

We add finally some technical  comments enlightening the main differences between the proof of our Theorem \ref{thm:main} and the analog Theorem \ref{thm:transport} for the transport equation proven in {\cite{modena2018non}, \cite{modena2019non}, \cite{modena2020convex}, \cite{brue2021positive}}. 

The proof is based on a convex integration scheme, inspired by the construction introduced by De Lellis and Sz\'ekelyhidi for the Euler equations (\cite{DeLellis:2009jh, DeLellis:2013im}, see also \cite{DeLellis:2012tz, Isett:2016to, Buckmaster:2017uz} for the applications of convex integration to the proof of Onsager's conjecture), coupled with the \emph{intermittency} or \emph{concentration} argument proposed by Buckmaster and Vicol in the mentioned paper {\cite{buckmaster2019nonuniqueness}} for the Navier-Stokes equations (see also \cite{buckmaster-colombo-vicol18} and \cite{burczak-mod-sze21}), and implemented in {\cite{modena2018non}} in the framework of the transport equation. 

From the technical point of view, the two main differences between the proof of our Theorem \ref{thm:main}  and Theorem \ref{thm:transport} in {\cite{modena2018non}, \cite{modena2019non}, \cite{modena2020convex}, \cite{brue2021positive}} are the following. 

On one side, our goal is not only to find an \emph{anomalous} solution to \eqref{eq:div-u-rhou}, where \emph{anomalous} could simply mean, for instance, that the l.h.s. of \eqref{eq:div-defect} is nonzero
 (in analogy to Theorem \ref{thm:transport} where densities are identically one at time $t=0$ and not identically one at time $t=1$, but the values of $\rho$ for times $t \in (0,1)$ are not prescribed). 
 
 On the contrary, we want to solve \emph{simultaneously} \eqref{eq:div-u-rhou} and \eqref{eq:div-defect}, for \emph{given} $\beta$ and $T$. To do this, we consider equation \eqref{eq:div-u-rhou} and \eqref{eq:div-defect} as a single system of PDEs to which we apply the convex integration scheme (a similar approach has been used in \cite{de2020non} to deal with globally dissipative solutions to the Euler equations). This has two main consequences:
 
\begin{itemize}
\item the \emph{error} (the analog of the Reynolds stress in convex integration schemes for the Euler equation) will not be a vector field (as in the proof of Theorem \ref{thm:transport} in {\cite{modena2018non}, \cite{modena2019non}, \cite{modena2020convex}, \cite{brue2021positive}}), but a $2 \times d$ matrix field. We need thus to introduce a suitable decomposition of $2 \times d$ matrix fields, which allows us to carry on the induction (see Lemma \ref{lemma:decomposition});

\item the perturbations $\nu,w$ we design to cancel the error at each step of the iteration (see \eqref{eq:perturbations-1}-\eqref{eq:perturbations-2})  must take care of both equations \eqref{eq:div-u-rhou} and \eqref{eq:div-defect}. In terms of the \emph{intermittency/concentration} argument, this requires $\nu$ and $\beta(\nu)$ to ``concentrate at the same rate'', and this is possible thanks to the bound on $\beta$ in \eqref{grwothcondition}. 
\end{itemize}  

The second main difference between the proof of our Theorem \ref{thm:main}  and Theorem \ref{thm:transport} in {\cite{modena2018non}, \cite{modena2019non}, \cite{modena2020convex}, \cite{brue2021positive}} is the following. The basic idea of convex integration (in fluid dynamics) is to construct a sequence $(\rho_q, u_q, R_q)$ of approximate solutions to \eqref{eq:div-u-rhou}-\eqref{eq:div-defect}, where $R_q$ is a (matrix valued) field representing the \emph{error} and converging to zero as $q \to \infty$. At each step, given $(\rho_q, u_q, R_q)$, one defines $\rho_{q+1} = \rho_q + \nu_{q+1}$ and $u_{q+1} = u_q + w_{q+1}$, where $\nu_{q+1}, w_{q+1}$ are suitable perturbations having, roughly speaking, the form
\begin{equation}
\label{eq:pert-intro}
\nu_{q+1} = a(R_q) \Theta_{q+1}, \qquad w_{q+1} = b(R_q) W_{q+1}
\end{equation}
which are designed to (almost) cancel the error $R_q$ and produce a new error $R_{q+1} \ll R_q$. In \eqref{eq:pert-intro}, $a,b$ are given functions, and $\Theta_{q+1}, W_{q+1}$ are the \emph{building blocks} of the construction, namely (fast oscillating and highly concentrated) solutions to \eqref{eq:div-u-rhou}. In our problem, however, equation \eqref{eq:div-defect}, which we also need to solve together with \eqref{eq:div-u-rhou}, is nonlinear in $\rho$ and therefore, in general,
\begin{equation*}
\beta(\rho_{q+1}) \neq \beta(\rho_q) + \beta(\nu_{q+1}) \neq \beta(\rho_q) + a(R_q) \beta( \Theta_{q+1}),
\end{equation*}
thus preventing us from merely applying standard techniques. To overcome this difficulty, we design, for every constant $a$, building blocks $\Theta_a, W_a$ (depending on $a$) such that $\beta(a \Theta_a) \approx a \beta(\Theta_a)$ (see Proposition \ref{prop:mikadosgeneral}), and then, by means of suitably designed cutoff functions, we decompose the torus $\T^d$ into small cubes, where the error $R_q$ is well approximated by constants and on each small cube we pick suitable building blocks in the families $\{\Theta_a\}_a, \{W_a\}_a$ (see Section \ref{sec:perturbations}).

\bigskip

We conclude this introduction by fixing some notation:
\begin{itemize}
\item $\tor = \mathbb{R}^d/\mathbb{Z}^d$ is the $d$-dimensional flat torus.
\item We denote by $\lbrace e_1,\dots, e_d\rbrace$ the set of standard basis vectors of $\mathbb{R}^d$.
\item We denote by $\mathbb{R}^{m\times n}$ the set of all $m\times n$-matrices.
\item If $R\in \mathbb{R}^{m\times n}, k\in \lbrace 1,\dots,m\rbrace, $ we denote by $R^m\in\mathbb{R}^n$ the $m$th row of $R.$
\item For $p\in [1,\infty]$, we denote by $p'$ its dual exponent.
\item We use the notation $C(A_1,\dots, A_n)$ to denote a constant which depends only on the quantities $A_1,\dots, A_n.$
\item $\mathbb{N} = \lbrace 0, 1, 2,\dots\rbrace.$
\end{itemize}
For a function $g\in C^\infty(\tor)$ and $\lambda\in\mathbb{N},$ we denote by $g_\lambda:\tor\rightarrow\mathbb{R}$ the $\frac{1}{\lambda}$ periodic function
$$g_\lambda (x) := g(\lambda x).$$
Notice that for every $k\in\mathbb{N}$, $r\in [1,\infty]$
$$\|D^kg_\lambda\|_{L^r(\tor)} = \lambda^k \|D^kg\|_{L^r(\tor)}.$$
For later reference we remark that for $k\in\mathbb{N},a\in\mathbb{R},\mu>0$ and a function $g\in C^\infty_c(\mathbb{R}^{d-1})$ it holds with $g_\mu\in C^\infty_c(\mathbb{R}^{d-1})$ defined by $g_\mu(x) := g(\mu x)$
\begin{align}\label{scaling}
\|D^k(\mu^a g_\mu)\|_{L^r(\mathbb{R}^{d-1})} = \mu^{a+k-\frac{d-1}{r}}\|D^kg\|_{L^r(\mathbb{R}^{d-1})}.
\end{align}

%
%
%
%

\section{Statement of the Main Proposition}
\label{sec:main-prop}

We start in this section the proof of Theorem \ref{thm:main}. Without loss of generality, we can assume that $\beta(0) = 0.$ Else, we replace $\beta$ by $\tilde{\beta} = \beta - \beta(0)$ and notice that $\dv(\tilde{\beta}(\rho)u) = \dv(\beta(\rho)u)$ if $\dv(u) = 0$. Following a well-established habit in fluid dynamics convex integration papers, we state here a \emph{Main Proposition} (the inductive step),  from which Theorem \ref{thm:main} easily follows. The proof of Proposition \ref{prop:main} will be done in Sections \ref{sec:perturbations-main-section}-\ref{sec:defect}-\ref{sec:proofofall}, whereas Sections  \ref{sec:preliminary} and \ref{ss:mikadosgeneral} are devoted respectively to the proof of some preliminary lemmas and to the definition of the (Mikado type) \emph{building blocks} of our construction.

\begin{prop}[Main Proposition]\label{prop:main}
There is a constant $M>0$ such that the following holds. Let $p\in (1,\infty)$ and $\tilde{p}\in \lbrack 1,\infty)$ such that \eqref{eq:d-1} holds. 
Let $h_0\in C^\infty(\tor; \R^d)$ and let $\rho_0:\tor\rightarrow\mathbb{R},$ $u_0:\tor\rightarrow\mathbb{R}^d$ and $R_0:\tor\rightarrow\mathbb{R}^{2\times d}$ be smooth functions satisfying the equation
\begin{equation}\label{eq:approximatesolution}
\begin{cases}
\dv(u_0)&=0,\\
\dv(\rho_0 u_0)&=-\dv R_0^1,\\
\dv(\beta(\rho_0)u_0-h_0)&=-\dv R_0^2.
\end{cases}
\end{equation}
Then for any $\delta>0$ and $h^\ast\in C^\infty(\tor; \R^d)$ with $\|h^\ast-h_0\|_{L^1(\tor)}\leq\frac{\delta}{4}$  there exist smooth functions  $\rho_1:\tor\rightarrow\mathbb{R},$ $u_1:\tor\rightarrow\mathbb{R}^d$ and $R_1:\tor\rightarrow\mathbb{R}^{2\times d}$ with
\begin{equation}\label{eq:equalitiesofR1}
\begin{cases}
\dv(u_1)&=0,\\
\dv(\rho_1 u_1)&=-\dv R_1^1,\\
\dv(\beta(\rho_1)u_1-h^\ast)&=-\dv R_1^2,
\end{cases}
\end{equation}
satisfying the estimates
\begin{align}
\|\rho_1-\rho_0\|_{L^p(\tor)}&\leq M\|R_0\|_{L^1(\tor)}^\frac{1}{p},\label{est:rho}\\
\|u_1-u_0\|_{L^{p'}(\tor)}&\leq M\|R_0\|_{L^1(\tor)}^\frac{1}{p'},\label{est:ulp}\\
\|u_1-u_0\|_{W^{1,\tilde{p}}(\tor)}&\leq\delta,\label{est:uw}\\
\|R_1\|_{L^1(\tor)}&\leq\delta.\label{est:R}
\end{align}
\end{prop}
 \begin{proof}[Proof of Theorem \ref{thm:main}, assuming Proposition \ref{prop:main}]
 Let $M$ be the constant from Proposition \ref{prop:main}. Let $\delta_q = 2^{-q}$ and let $(h_q)_{q\in\mathbb{N}}$ be a sequence of smooth functions such that $h_q\rightarrow h$ in $L^1(\tor)$ and \mbox{$\|h_{q+1}-h_q\|_{L^1(\tor)}\leq\frac{\delta_{q+1}}{4}.$} We construct a sequence $(\rho_q,u_q,R_q)$ of solutions to 
\begin{equation}\label{eq:approximatesolutioniteration}
\begin{cases}
\dv(u_q)&=0,\\
\dv(\rho_q u_q)&=-\dv R_q^1,\\
\dv(\beta(\rho_q)u_q-h_q)&=-\dv R_q^2
\end{cases}
\end{equation}
as follows. We let $\rho_0 = 0,$ $u_0 = 0,$ $R_0^1=0$ and $R_0^2 = h_0.$ With this definition, the tuple $(\rho_0,u_0,R_0,h_0)$ satisfies \eqref{eq:approximatesolution}. Assume now $(\rho_q, u_q, R_q, h_q)$ is defined, satisfying \eqref{eq:approximatesolutioniteration}. Then we obtain $(\rho_{q+1},u_{q+1}, R_{q+1})$ 
 by applying Proposition \ref{prop:main} to $(\rho_q, u_q, R_q)$ with $h^\ast = h_{q+1},$ $\delta =\delta_{q+1}.$ This tuple satisfies \eqref{eq:approximatesolutioniteration} with $q$ replaced by $q+1$ and by \eqref{est:rho} -- \eqref{est:R}  it satisfies for $q\geq 1$ the (inductive) estimates
\begin{align*}
\|\rho_{q+1}-\rho_q\|_{L^p(\tor)}&\leq  M\|R_q\|_{L^1(\tor)}^\frac{1}{p}\leq  M \delta_{q}^\frac{1}{p},\\
\|u_{q+1}-u_q\|_{L^{p'}(\tor)}&\leq  M\|R_q\|_{L^1(\tor)}^\frac{1}{p'}\leq  M\delta_{q}^\frac{1}{p'},\\
\|u_{q+1}-u_q\|_{W^{1,\tilde{p}}(\tor)}&\leq\delta_{q+1},\\
\|R_{q+1}\|_{L^1(\tor)}&\leq\delta_{q+1}.
\end{align*}
These estimates, together with the assumption $C|\tau|\leq\beta(\tau)\leq D|\tau|$ for large values of $|\tau|$ show that there exist $\rho\in L^p(\tor)$ and $u\in W^{1,\tilde{p}}(\tor)\cap L^{p'}(\tor)$ such that
\begin{align*}
\rho_q &\rightarrow \rho \text{ in } L^p(\tor),\\
u_q &\rightarrow u \text{ in } W^{1,\tilde{p}}(\tor)\cap L^{p'}(\tor),
\end{align*}
being a weak solution to \eqref{eq:div-u-rhou} and \eqref{eq:div-defect}.
 \end{proof}

\section{Preliminary tools}
\label{sec:preliminary}

We state here two preliminary lemmas. The first one concerns the existence of an \emph{antidivergence operator} which allows a gain of a factor $\lambda^{-1}$ when applied to the product of two maps, one of the two being oscillating with period $\lambda^{-1}$. For the proof we refer to \cite[Lemma 2.3]{modena2018non}.

\begin{lemma}[Antidivergence operator]\label{lemma:antidiv}
Let $\lambda\in\mathbb{N} \setminus \{0\}$ and $f,g:\tor\rightarrow\mathbb{R}$ be smooth functions with
\begin{align*}
\int_{\tor}fg_\lambda \, dx = \int_{\tor}g \, dx= 0.
\end{align*}
Then there exists a smooth vector field $u:\tor\rightarrow\mathbb{R}^d$ such that $\dv u =fg_\lambda$ and for every $p\in[1,\infty]$ and $k\in\mathbb{N}$
\begin{align*}
\|D^ku\|_{L^p(\tor)}\leq C_{k,p}\lambda^{k-1}\|f\|_{C^{k+1}(\tor)}\|g\|_{W^{k,p}(\tor)}.
\end{align*}
We write $u = \mathcal{R}(fg_\lambda).$
\end{lemma}

The second lemma is a decomposition of $2 \times d$-matrix fields, which allows us to deal with the second equation in \eqref{eq:div-u-rhou} and \eqref{eq:div-defect} as a single system of PDEs.

\begin{lemma}\label{lemma:decomposition}
Let $R:\tor\rightarrow\mathbb{R}^{2\times d}$ be a smooth function. There exist smooth functions $g_j:\tor\rightarrow\mathbb{R}^d$  and $z_j\in\lbrace e_1,\dots,e_d\rbrace,j\in\lbrace 1,\dots, 2d\rbrace,$ with the following property. We have the decomposition
\begin{equation*}
R(x)=\sum_{j=1}^{2d}\begin{pmatrix}
g_j(x)\\(-1)^{j+1}g_j(x)
\end{pmatrix}
\otimes z_j
\end{equation*}
and the estimate $\|g_j\|_{L^1(\tor)}\leq \|R\|_{L^1(\tor)}$ as well as the pointwise estimate $|g_j(x)|\leq |R(x)|.$
\end{lemma}

\begin{rem}
Due to  $z_j\in\lbrace e_1,\dots,e_d\rbrace$, one single summand of the above representation has the form
\begin{equation*}
\begin{pmatrix}
g_j(x)\\ (-1)^{j+1}g_j(x)
\end{pmatrix}
\otimes z_j =
\begin{pmatrix}
0 & \dots & g_j(x) & \dots 0\\
0 & \dots & (-1)^{j+1}g_j(x) & \dots 0
\end{pmatrix}.
\end{equation*}
\end{rem}

\begin{proof}
Let us first write $R$ in the standard basis as
$$R(x)=\begin{pmatrix}
R^{1,1}(x) & \dots & R^{1,d}(x)\\
R^{2,1}(x) & \dots & R^{2,d}(x)
\end{pmatrix}.$$ Let \mbox{$k\in\lbrace 1,\dots, d\rbrace.$} We define 
\begin{align*}
g_{2k-1}(x) = \frac{R^{1,k}(x) + R^{2,k}(x)}{2},\hspace{0,3cm} &z_{2k-1} = e_{k},\\
g_{2k}(x) = \frac{R^{1,k}(x) - R^{2,k}(x)}{2}, \hspace{0,3cm}&z_{2k} = e_{k}.
\end{align*}
Then we have
\begin{align*}
 R^{1,k} e_{k}&=g_{2k-1}z_{2k-1} + g_{2k}z_{2k},\\
R^{2,k} e_k&=(-1)^{2k-1+1} g_{2k-1}z_{2k-1}+ (-1)^{2k+1} g_{2k}z_{2k},
\end{align*}
hence this yields the desired decomposition.
\end{proof}

\section{The Building Blocks}
\label{ss:mikadosgeneral}

In this section we construct the building blocks of our construction. They are ``concentrated Mikado type'' densities and vector fields (see \cite{SzekelyhidiJr:2016tp} for the origin of \emph{Mikado flows}), with the additional property that they are designed in order to match in a suitable way both with the linear (in $\rho$) equation \eqref{eq:div-u-rhou} and with the nonlinear (because of the presence of the renormalization function $\beta$) equation \eqref{eq:div-defect}.

\begin{prop}\label{prop:mikadosgeneral}
Let $0\neq a\in\mathbb{R},$ $\sigma\in\left\lbrace-1,1\right\rbrace$ and $\mu_0=\mu_0(a)>0$ such that \eqref{grwothcondition} holds for all $|\tau|\geq  |a|^\frac{1}{p}\mu_0^{\frac{d-1}{p}}.$ For all $k\in \lbrace 1,\dots,d\rbrace,$ $\zeta>0,$ $\mu\geq\mu_0,$  we have a \emph{Mikado Density}
$\Theta_{a,k,\zeta,\mu,\sigma}\in C^\infty(\tor)$ and  a \emph{Mikado Field}  $W_{a,k,\zeta,\mu,\sigma}\in C^\infty(\tor;\mathbb{R}^d)$ satisfying the following properties:
\begin{equation}\label{eq:propertieofmikadosgeneral}
\begin{cases}
\dv W_{a,k,\zeta,\mu,\sigma} &= 0,\\
\dv \Theta_{a,k,\zeta,\mu,\sigma}W_{a,k,\zeta,\mu,\sigma} &= 0,\\
\int_{\mathbb{T}^{d}}W_{a,k,\zeta,\mu,\sigma}\, dx & = 0,\\
\left|\int_{\mathbb{T}^{d}}\Theta_{a,k,\zeta,\mu,\sigma}W_{a,k,\zeta,\mu,\sigma}\, dx - e_k \right|&< \zeta,\\
\left|\int_{\mathbb{T}^{d}}\beta(|a|^\frac{1}{p}\Theta_{a,k,\zeta,\mu,\sigma})W_{a,k,\zeta,\mu,\sigma}\, dx - \sigma |a|^\frac{1}{p} e_k\right| &<\zeta
\end{cases}
\end{equation}
and support contained in $\left[(0,\frac{1}{\mu})^{k-1}\times\mathbb{R}\times (0,\frac{1}{\mu})^{d-k}\right] + \mathbb{Z}^d.$
Furthermore, there exists a constant $M_0>0$, depending on the renormalization map $\beta$, but  independent of $a,k,\zeta,\mu$ and $\sigma$, such that the following estimates hold for all $r\in [1,\infty]$
\begin{align}
\|\Theta_{a,k,\zeta,\mu,\sigma}\|_{L^r(\tor)}&\leq M_0\mu^{\frac{d-1}{p}-\frac{d-1}{r}},\label{eq:thetainlrgeneral}\\
\|W_{a,k,\zeta,\mu,\sigma}\|_{L^r(\tor)}&\leq M_0\mu^{\frac{d-1}{p'}-\frac{d-1}{r}}, \label{eq:winlrgeneral}\\
\|\Theta_{a,k,\zeta,\mu, \sigma}W_{a,k,\zeta,\mu,\sigma}\|_{L^r(\tor)}&\leq M_0^2\mu^{d-1-\frac{d-1}{r}}.\nonumber
\end{align}
Furthermore, it holds for all $m\in\mathbb{N}$ and $r\in\lbrack 1,\infty\rbrack$
\begin{align*}
\|D^m\Theta_{a,k,\zeta,\mu,\sigma}\|_{L^r(\mathbb{T}^d)}&\leq C(a,\zeta,m)\mu^{\frac{d-1}{p}+m-\frac{d-1}{r}},\\
\|D^mW_{a,k,\zeta,\mu,\sigma}\|_{L^r(\mathbb{T}^d)}&\leq C(a,\zeta,m) \mu^{\frac{d-1}{p'}+m-\frac{d-1}{r}}.
\end{align*}
In particular, we have that
\begin{align}
\|\Theta_{a,k,\zeta,\mu,\sigma}\|_{L^p(\tor)} &\leq M_0,\label{eq:thetainlpgeneral}\\
\|W_{a,k,\zeta,\mu,\sigma}\|_{L^{p'}(\tor)}&\leq M_0,\label{eq:winlp'general}\\
\|\Theta_{a,k,\zeta,\mu,\sigma}W_{a,k,\zeta,\mu}\|_{L^1(\tor)}&\leq M_0^2,  \label{eq:thetawproductgeneral}\\
\|DW_{a,k,\zeta,\mu,\sigma}\|_{L^{\tilde{p}}(\tor)}&\leq C(a,\zeta,m)\mu^{-\gamma}\label{eq:dwlptildegeneral}
\end{align}
with $\gamma = \frac{d-1}{\tilde{p}}-\frac{d-1}{p'}-1>0$ because of \eqref{eq:d-1}.
\end{prop}

Before we can give the definition of our building blocks, the Mikado Fields and Densities, we need to construct a family of auxiliary functions, which will be the content of the next lemma.
\begin{lemma}\label{lemma:construction1}
Let $a, \sigma, \mu_0$ be as in the statement of Proposition \ref{prop:mikadosgeneral}.
 For every $\zeta>0$, $\mu\geq\mu_0$ there exist two functions $\Psi_{a,\zeta,\mu,\sigma}$, $\tilde{\Psi}_{a,\zeta,\mu,\sigma}\in C^\infty(\mathbb{R}^{d-1})$  with support in $(0,\frac{1}{\mu})^{d-1}$  such that
\begin{align}
\left|\int_{\mathbb{R}^{d-1}}\Psi_{a,\zeta,\mu,\sigma}\tilde{\Psi}_{a,\zeta,\mu,\sigma}\, dx -1\right| &< \zeta,\label{eq:meanvaluenuw}\\
\left|\int_{\mathbb{R}^{d-1}}\beta(|a|^\frac{1}{p}\Psi_{a,\zeta,\mu,\sigma})\tilde{\Psi}_{a,\zeta,\mu,\sigma}\, dx - \sigma |a|^\frac{1}{p}\right|&<\zeta\text{ and}\label{eq:meanvaluebeta}\\
\int_{\mathbb{R}^{d-1}}\tilde{\Psi}_{a,\zeta,\mu,\sigma}\, dx&= 0, \label{eq:meanvaluew-original}
\end{align}
satisfying for any $r\in [1,\infty]$ the estimates
\begin{align}\label{eq:scaling1}
\|\Psi_{a,\zeta,\mu,\sigma}\|_{L^r(\mathbb{R}^{d-1})}&\leq M_1\mu^{\frac{d-1}{p}-\frac{d-1}{r}},\\
\|\tilde{\Psi}_{a,\zeta,\mu,\sigma}\|_{L^r(\mathbb{R}^{d-1})}&\leq M_1\mu^{\frac{d-1}{p'}-\frac{d-1}{r}}\text{ and}\label{eq:scaling5}\\
\|\Psi_{a,\zeta,\mu,\sigma}\tilde{\Psi}_{a,\zeta,\mu,\sigma}\|_{L^r(\mathbb{R}^{d-1})}&\leq M_1^2\mu^{d-1-\frac{d-1}{r}}\label{eq:scaling2}
\end{align}
with a constant $M_1$ depending on the renormalization map $\beta$, but independent of $a,\zeta,\mu$ and $\sigma.$ For $m\geq1,$ it holds
\begin{align}
\|D^m\Psi_{a,\zeta,\mu,\sigma}\|_{L^r(\mathbb{R}^{d-1})}&\leq C(a,\zeta,m)\mu^{\frac{d-1}{p}+m-\frac{d-1}{r}},\label{eq:scaling3}\\
\|D^m\tilde{\Psi}_{a,\zeta,\mu, \sigma}\|_{L^r(\mathbb{R}^{d-1})}&\leq C(a,\zeta,m)\mu^{\frac{d-1}{p'}+m-\frac{d-1}{r}}\label{eq:scaling4}
\end{align}
with constants $C(a,\zeta,m)>0.$
\end{lemma}
\begin{proof}
Let $a\neq 0,$ $\sigma\in\left\lbrace-1,1\right\rbrace$ and fix $\zeta>0$ and $\mu\geq\mu_0(a).$ Let $P_1=(\frac{3}{4},\dots,\frac{3}{4}), P_2 = (\frac{1}{4},\frac{3}{4},\dots,\frac{3}{4}), P_3 =(\frac{3}{4},\frac{1}{4},\frac{3}{4},\dots,\frac{3}{4})\in\mathbb{R}^{d-1}.$
We consider first two (non-continuous) functions $\varphi,\tilde{\varphi}:\mathbb{R}^{d-1}\rightarrow\mathbb{R}$ defined by

\begin{equation}
\label{eq:varphi}
\varphi = \chi_{B_{\frac{1}{8}}(P_1)} - \chi_{B_{\frac{1}{8}}(P_2)}
\end{equation}
and


\begin{equation}
\label{eq:varphitilde}
\tilde{\varphi} = \alpha_1\chi_{B_{\frac{1}{8}}(P_1)}+\alpha_2\chi_{B_{\frac{1}{8}}(P_2)} + \alpha_3 \chi_{B_{\frac{1}{8}}(P_3)}
\end{equation}
with $\alpha=(\alpha_1, \alpha_2,\alpha_3)\in\mathbb{R}^3$ to be determined.
We will choose $\alpha = \alpha(a,\mu,\sigma)$ depending on $a,\mu$ and $\sigma$ such that the concentrated versions of these functions satisfy equations \eqref{eq:meanvaluenuw} and \eqref{eq:meanvaluebeta} exactly. To be precise, this means that the concentrated functions
\begin{align*}
\psi(x) := \mu^{\frac{d-1}{p}}\varphi(\mu x) \text{ and}\\
\tilde{\psi}(x) := \mu^{\frac{d-1}{p'}}\tilde{\varphi}(\mu x)
\end{align*}
satisfy
\begin{align}
\int_{\mathbb{R}^{d-1}}\psi \tilde \psi\, dx &= 1 \text{ and }\label{eq:meanvaluenuwexact}\\
\int_{\mathbb{R}^{d-1}}\beta(|a|^\frac{1}{p}\psi)\tilde \psi \, dx&= \sigma|a|^\frac{1}{p}.\label{eq:meanvaluebetaexact}
\end{align}
Equations \eqref{eq:meanvaluenuwexact} and \eqref{eq:meanvaluebetaexact} can be rewritten as a system of linear equations in the unknowns $\alpha_1, \alpha_2$, by calculating the integrals on the left-hand side. Indeed, we need to solve

\begin{equation}
\label{eq:alpha}
\begin{pmatrix}
1 & -1  \\
\frac{\beta\left(|a|^\frac{1}{p}\mu^{\frac{d-1}{p}}\right)}{\mu^{\frac{d-1}{p}}}  & \frac{\beta\left(-|a|^\frac{1}{p}\mu^{\frac{d-1}{p}}\right)}{\mu^{\frac{d-1}{p}}}
\end{pmatrix}
\cdot
\begin{pmatrix}
\alpha_1\\
\alpha_2
\end{pmatrix}
=
\begin{pmatrix}
\frac{8^{d-1}}{\omega_{d-1}}\\
\frac{8^{d-1}\sigma|a|^\frac{1}{p}}{\omega_{d-1}}
\end{pmatrix},
\end{equation}
where $\omega_{d-1}$ is the Lebesgue measure of $B_1(0)\subset\mathbb{R}^{d-1}.$ 
The above matrix is invertible (recall that $\beta$ is nonnegative for the choice of $\mu \geq \mu_0$) and the values $\alpha_1,\alpha_2$ are given by
\begin{align*}
\begin{pmatrix}
\alpha_1\\
\alpha_2
\end{pmatrix}
&=
\frac{1}{\beta\left(-|a|^\frac{1}{p}\mu^{\frac{d-1}{p}}\right)+\beta\left(|a|^\frac{1}{p}\mu^{\frac{d-1}{p}}\right)}
\begin{pmatrix}
\beta\left(-|a|^\frac{1}{p}\mu^{\frac{d-1}{p}}\right) & \mu^{\frac{d-1}{p}}\\
-\beta\left(|a|^\frac{1}{p}\mu^{\frac{d-1}{p}}\right) & \mu^{\frac{d-1}{p}}
\end{pmatrix}\cdot
\begin{pmatrix}
\frac{8^{d-1}}{\omega_{d-1}}\\
\frac{8^{d-1}\sigma|a|^\frac{1}{p}}{\omega_{d-1}}
\end{pmatrix}\\
&= \frac{8^{d-1}}{\omega_{d-1}\left\lbrack\beta\left(-|a|^\frac{1}{p}\mu^{\frac{d-1}{p}}\right)+\beta\left(|a|^\frac{1}{p}\mu^{\frac{d-1}{p}}\right)\right\rbrack}
\begin{pmatrix}
\beta\left(-|a|^\frac{1}{p}\mu^{\frac{d-1}{p}}\right)+\mu^{\frac{d-1}{p}}\sigma|a|^\frac{1}{p}\\
-\beta\left(|a|^\frac{1}{p}\mu^{\frac{d-1}{p}}\right)+\mu^{\frac{d-1}{p}}\sigma|a|^\frac{1}{p}
\end{pmatrix}.
\end{align*} 
This fixes $\alpha_1, \alpha_2$ (depending on $a,\mu,\sigma$) in the definition of $\tilde{\varphi}.$ Now using again inequality \eqref{grwothcondition} (which holds for $\mu \geq \mu_0$), we see that
\begin{align*}
|\alpha_1|, |\alpha_2|\leq \frac{8^{d-1}}{\omega_{d-1}}  \left( 1 + \frac{1}{D} \right).
\end{align*}
We then define $\alpha_3 := - (\alpha_1 + \alpha_2)$ to ensure that
\begin{equation}
\label{eq:meanvaluew}
\int_{\R^{d-1}} \tilde \varphi \, dx = 0.
\end{equation}
Notice that 
$$|\alpha_3|\leq  \frac{2 \cdot 8^{d-1}}{\omega_{d-1}} \left( 1 + \frac{1}{D} \right).$$
This means that $\varphi,\tilde{\varphi}$ are bounded (in $L^\infty$) with a bound dependent on the constants $D$ (thus on the renormalization map $\beta$), but independent of $a,$ $\zeta,$ $\mu$ and $\sigma,$ namely
$$ |\varphi|, |\tilde \varphi| \leq M_1 := \max\left\lbrace 1, \frac{2 \cdot 8^{d-1}}{\omega_{d-1}} \left( 1 + \frac{1}{D} \right) \right\rbrace.$$
Now, since we want the perturbations to be smooth functions, we will use a mollification of the maps $\varphi$ and $\tilde{\varphi}$ introduced in \eqref{eq:varphi}, \eqref{eq:varphitilde}. We set
\begin{align*}
\Phi &:= \eta_{\ell}\ast \varphi,\\
\tilde{\Phi} &:= \eta_{\ell}\ast \tilde{\varphi}.
\end{align*}
Here, $\eta$ is a standard mollifier and we fix $\ell = \ell(a,\zeta)$ so small such that (recall that $p,p' \in (1, \infty)$ by assumption)
\begin{align*}
\|\Phi-\varphi\|_{L^p(\mathbb{R}^{d-1})},\|\tilde{\Phi}-\tilde{\varphi}\|_{L^{p'}(\mathbb{R}^{d-1})}&<  \frac{\zeta}{2M_1}\text{ and}\\
|a|^\frac{1}{p'}\|\Phi-\varphi\|_{L^p(\mathbb{R}^{d-1})},
|a|^\frac{1}{p}\|\tilde{\Phi}-\tilde{\varphi}\|_{L^{p'}(\mathbb{R}^{d-1})}&<\frac{\zeta}{2M_1\operatorname{Lip}(\beta)}
\end{align*}
and such that the supports of $\Phi,\tilde{\Phi}$ still consist of two (or three, respectively) disjoint balls contained in $(0,1)^{d-1}.$ Notice that we still have $|\Phi|,|\tilde{\Phi}|\leq M_1.$ Now, the desired functions $\Psi_{a,\zeta,\mu,\sigma},\tilde{\Psi}_{a,\zeta,\mu,\sigma}$ are given by the concentrated functions
\begin{align*}
\Psi_{a,\zeta,\mu,\sigma}(x) &:= \mu^{\frac{d-1}{p}}\Phi(\mu x) \text{ and}\\
\tilde{\Psi}_{a,\zeta,\mu,\sigma}(x) &:= \mu^{\frac{d-1}{p'}}\tilde{\Phi}(\mu x).
\end{align*}
We abbreviate $\Psi :=\Psi_{a,\zeta,\mu,\sigma},$ $\tilde{\Psi}:=\tilde{\Psi}_{a,\zeta,\mu,\sigma}.$
We show that they have the desired properties. We clearly have $\supp(\Psi),$ $\supp(\tilde{\Psi})\subset(0,\frac{1}{\mu})^{d-1}$ by the concentration. Inequality \eqref{eq:scaling1} is obtained by the scaling property \eqref{scaling} and the bound on $\Phi$:

\begin{align*}
\|\Psi\|_{L^r(\mathbb{R}^{d-1})} =  \mu^{\frac{d-1}{p}-\frac{d-1}{r}}\|\Phi\|_{L^r(\mathbb{R}^{d-1})} =  \mu^{\frac{d-1}{p}-\frac{d-1}{r}}\|\Phi\|_{L^r([0,1]^{d-1})}\leq M_1  \mu^{\frac{d-1}{p}-\frac{d-1}{r}}.
\end{align*}
A similar estimate shows \eqref{eq:scaling5} and \eqref{eq:scaling2}.
Now, using the scaling property \eqref{scaling} and the bound on $\Phi,\tilde{\Phi}, \varphi,\tilde{\varphi}$ again and property \eqref{eq:meanvaluenuwexact},  we obtain 
\begin{align*}
\left|\int_{\mathbb{R}^{d-1}}\Psi\tilde{\Psi}\, dx - 1\right| &= \left|\int_{\mathbb{R}^{d-1}}\Psi\tilde{\Psi}-\psi\tilde \psi\, dx\right|\leq\left|\int_{\mathbb{R}^{d-1}}\Psi\tilde{\Psi}-\Psi\tilde{\psi}\, dx\right|+\left|\int_{\mathbb{R}^{d-1}}\Psi\tilde \psi-\psi\tilde \psi\, dx\right|\\
&\leq\|\Psi\|_{L^p(\mathbb{R}^{d-1})}\|\tilde{\Psi}-\tilde \psi\|_{L^{p'}(\mathbb{R}^{d-1})} + \|\tilde \psi\|_{L^{p'}(\mathbb{R}^{d-1})}\|\Psi-\psi\|_{L^p(\mathbb{R}^{d-1})}\\
&\leq M_1 \|\tilde{\Phi}-\tilde{\varphi}\|_{L^{p'}(\mathbb{R}^{d-1})} + M_1\|\Phi-\varphi\|_{L^p(\mathbb{R}^{d-1})}\\
&<\zeta
\end{align*}
which gives estimate \eqref{eq:meanvaluenuw}.
Furthermore, using the Lipschitz continuity of $\beta$ and the assumption $\beta(0) = 0,$ property \eqref{eq:meanvaluebetaexact} and  again the scaling property of the concentration,  we estimate similarly
\begin{align*}
\left|\int_{\mathbb{R}^{d-1}}\beta(|a|^\frac{1}{p}\Psi)\tilde{\Psi}\, dx-\sigma|a|^\frac{1}{p}\right|&=\left|\int_{\mathbb{R}^{d-1}}\beta(|a|^\frac{1}{p}\Psi)\tilde{\Psi}-\beta(|a|^\frac{1}{p}\psi)\tilde \psi\, dx\right|\\
&\leq\|\beta(|a|^\frac{1}{p}\Psi)\|_{L^p(\mathbb{R}^{d-1})}\|\tilde{\Psi}-\tilde \psi\|_{L^{p'}(\mathbb{R}^{d-1})}\\
&\hspace{0,1cm} +\|\tilde \psi\|_{L^{p'}(\mathbb{R}^{d-1})}\|\beta(|a|^\frac{1}{p}\Psi)-\beta(|a|^\frac{1}{p}\psi)\|_{L^p(\mathbb{R}^{d-1})}\\
\text{(since $\beta(0)= 0$)} &\leq\operatorname{Lip}(\beta)|a|^\frac{1}{p}\|\Psi\|_{L^p(\mathbb{R}^{d-1})}\|\tilde{\Psi}-\tilde{\psi}\|_{L^{p'}(\mathbb{R}^{d-1})}\\
&\hspace{0,1cm} + \operatorname{Lip}(\beta)|a|^\frac{1}{p}\|\tilde{\psi}\|_{L^{p'}(\mathbb{R}^{d-1})}\|\Psi-\psi\|_{L^p(\mathbb{R}^{d-1})}\\
\text{(using concentration property \eqref{scaling})} &\leq\operatorname{Lip}(\beta)|a|^\frac{1}{p}\|\Phi\|_{L^p(\mathbb{R}^{d-1})}\|\tilde{\Phi}-\tilde{\varphi}\|_{L^{p'}(\mathbb{R}^{d-1})}\\
&\hspace{0,1cm} + \operatorname{Lip}(\beta)|a|^\frac{1}{p}\|\tilde{\varphi}\|_{L^{p'}(\mathbb{R}^{d-1})}\|\Phi-\varphi\|_{L^p(\mathbb{R}^{d-1})}\\
&\leq M_1\operatorname{Lip}(\beta)|a|^\frac{1}{p}\left(\|\tilde{\Phi}-\tilde{\varphi}\|_{L^{p'}(\mathbb{R}^{d-1})} + \|\Phi-\varphi\|_{L^p(\mathbb{R}^{d-1})}\right)\\
&<\zeta,
\end{align*}
showing \eqref{eq:meanvaluebeta}. Condition \eqref{eq:meanvaluew-original} follows from \eqref{eq:meanvaluew}.
Finally, we show \eqref{eq:scaling3} and \eqref{eq:scaling4}. Going back to the definition of $\Phi,$ we get

\begin{align*}
\|D^m\Psi\|_{L^r(\mathbb{R}^{d-1})}&=\mu^{\frac{d-1}{p}+m-\frac{d-1}{r}}\|D^m\Phi\|_{L^r(\mathbb{R}^{d-1})}=\mu^{\frac{d-1}{p}+m-\frac{d-1}{r}}\|D^m(\eta_\ell)\ast\varphi\|_{L^r(\mathbb{R}^{d-1})}\\
&=\frac{\mu^{\frac{d-1}{p}+m-\frac{d-1}{r}}}{\ell^m}\|(D^m\eta)_\ell\ast\varphi\|_{L^r(\mathbb{R}^{d-1})}= C(a,\zeta,m)\mu^{\frac{d-1}{p}+m-\frac{d-1}{r}}
\end{align*}
since $\ell$ only depends on $a$ and $\zeta$ and because $\eta$ was a fixed mollifier. The same computation can be done for the norm  of $D^m\tilde{\Psi}$ in $L^r(\mathbb{R}^{d-1}).$ This shows \eqref{eq:scaling3} and \eqref{eq:scaling4}.
\end{proof}

\begin{proof}[Proof of Proposition \ref{prop:mikadosgeneral}]
Fix $k\in\lbrace 1,\dots, d\rbrace$, $a \neq 0$, $\sigma \in \{-1,+1\}$, $\zeta>0$ and $\mu\geq\mu_0.$ We apply Lemma \ref{lemma:construction1} and get some $\Psi_{a,\zeta,\mu,\sigma}, \tilde{\Psi}_{a,\zeta,\mu,\sigma}\in C^\infty(\mathbb{R}^{d-1}).$ We define the (non-periodic) Mikado Density $\tilde{\Theta}_{a,k,\zeta,\mu,\sigma}:\mathbb{R}^d\rightarrow\mathbb{R}$ and Mikado Field $\tilde{W}_{a,k,\zeta,\mu,\sigma}:\mathbb{R}^{d}\rightarrow\mathbb{R}^d$ as 
\begin{align*}
\tilde{\Theta}_{a,k,\zeta,\mu,\sigma}(x) &= \Psi_{a,\zeta,\mu,\sigma}(x_1,\dots,x_{k-1},x_{k+1},\dots,x_d),\\
\tilde{W}_{a,k,\zeta,\mu,\sigma}(x) &= \tilde{\Psi}_{a,\zeta,\mu,\sigma}(x_1,\dots,x_{k-1},x_{k+1},\dots,x_d)e_k.
\end{align*}
Now we define the Mikado Density $\Theta_{a,k,\zeta,\mu,\sigma}:\tor\rightarrow\mathbb{R}$ and the Mikado Field \mbox{$W_{a,k,\zeta,\mu,\sigma}:\tor\rightarrow\mathbb{R}^d$} as  the $1$-periodic extensions of $\tilde{\Theta}_{a,k,\zeta,\mu,\sigma}, \tilde{W}_{a,k,\zeta,\mu,\sigma},$ respectively.
This is well-defined because $\operatorname{supp}(\Psi_{a,\zeta,\mu,\sigma}),$ $\operatorname{supp}(\tilde{\Psi}_{a,\zeta,\mu,\sigma})\subset (0,1)^{d-1}$ and $\tilde{\Theta}_{a,k,\zeta,\mu,\sigma},$ $\tilde{W}_{a,k,\zeta,\mu,\sigma}$ are constant in direction $e_k.$
By the properties of $\Psi_{a,\zeta,\mu,\sigma}$ and $\tilde{\Psi}_{a,\zeta,\mu,\sigma},$ it is easy to check that $$\supp(\Theta_{a,k,\zeta,\mu,\sigma}),\supp(W_{a,k,\zeta,\mu,\sigma})\subset \left[\left(0,\frac{1}{\mu}\right)^{k-1}\times\mathbb{R}\times \left(0,\frac{1}{\mu}\right)^{d-k}\right]+\mathbb{Z}^d$$ and that the desired estimates are satisfied with $M_0=M_1$ from Lemma \ref{lemma:construction1}.
\end{proof}

\begin{rem}
\label{rem:sign-changing-density}
Going through the proofs of Proposition \ref{prop:mikadosgeneral} and Lemma \ref{lemma:construction1}, one can see that the densities $\Theta_{a,k,\zeta,\mu,\sigma}$ \emph{must} be chosen in such a way that they do not have a definite sign. Indeed, the function used to define $\Theta_{a,k,\zeta,\mu,\sigma}$ is  essentially the map $\varphi$ (which takes both positive and negative values) introduced in \eqref{eq:varphi}, whereas the one used to define $W_{a,k,\zeta,\mu,\sigma}$ is the map $\tilde \varphi$ introduced in \eqref{eq:varphitilde}. If $\varphi$ were always nonnegative and the renormalization function $\beta$ is, say, just $\beta(\tau) = |\tau|$, then equations \eqref{eq:meanvaluenuwexact} and \eqref{eq:meanvaluebetaexact} could be achieved simultaneously only for $\sigma = +1$, but not for $\sigma = -1$ (and, in particular, the inverse of the matrix in \eqref{eq:alpha} and the related coefficients $(\alpha_1, \alpha_2)$ would not be uniformly bounded). On the other side, we need to construct building blocks for which \eqref{eq:meanvaluebetaexact} (and correspondingly the last equation in \eqref{eq:propertieofmikadosgeneral}) holds also for $\sigma = -1$, since, in the decomposition provided by Lemma \ref{lemma:decomposition}, the first and the second row of the error matrix have different sign (for $j$ odd, for any fixed $x$).
%
%
%
%
%
\end{rem}

\section{The perturbations}
\label{sec:perturbations-main-section}

We start  in this section the proof of the Main Proposition \ref{prop:main}. We define first of all the constant $M$ by
\begin{equation}
\label{eq:constant-m}
M:= M_0\max\lbrace (3d)^\frac{1}{p}, (3d)^\frac{1}{p'}\rbrace,
\end{equation}
where $M_0$ is the constant appearing in the statement of Proposition \ref{prop:mikadosgeneral}. Observe that $M_0$ (and thus $M$) depends only on the renormalization map $\beta$, and on nothing else. We assume then that $p,\tilde p$, $h_0, h^*$, $(\rho_0, u_0, R_0)$ and $\delta$ are given as in the statement of Proposition \ref{prop:main}, and we define in Section \ref{sec:perturbations} (see in particular \eqref{eq:perturbations-1}, \eqref{eq:perturbations-2}, \eqref{eq:perturbations-3}) the new density $\rho_1$ and the new vector field $u_1$. In Section \ref{ss:estimates-perturbation} we prove some estimates for $\rho_1,u_1$. The new defect $R_1$ will then be defined subsequently, in Section \ref{sec:defect}. 

\subsection{Definition of the perturbations}\label{sec:perturbations}
In this section, we are going to define the perturbations that will lead to $\rho_1$ and $u_1$ in \mbox{Proposition \ref{prop:main}.} For this definition, we will introduce in the following several parameters. They are (in the order in which they will appear in the subsections below)
\begin{align*}
\varepsilon>0 \text{ with } \frac{1}{\varepsilon}\in\mathbb{N}&: \text{ fineness of the decomposition of } \tor \text{ into cubes } Q\in\mathcal{Q}, && \text{Section \ref{ss:partition}} \\
\zeta>0&: \text{ approximation parameter for the mean values of the Mikados}, && \text{Section \ref{ss:choice-mikado}} \\
\mu \gg 1&:\text{ concentration}, && \text{Section \ref{ss:choice-mikado}} \\
\alpha\in (0,1)&:\text{ approximation parameter for the cutoff functions}, && \text{Section \ref{ss:cutoffs}}  \\
\lambda \in \frac{1}{\varepsilon}\mathbb{N}&: \text{ oscillation}, && \text{Section \ref{ss:def-perturbations}}.
\end{align*}
The values of these parameters will be fixed in Section \ref{sec:proofofall}. Note that $\delta>0$ is no parameter since it is already fixed by the assumptions of Proposition \ref{prop:main}.

\subsubsection{Partition into cubes of edge $\e$}
\label{ss:partition}

We decompose $R_0$ with Lemma \ref{lemma:decomposition}, i.e. we consider the representation
\begin{equation*}
R(x)=\sum_{j=1}^{2d}\begin{pmatrix}
g_j(x)\\(-1)^{j+1}g_j(x)
\end{pmatrix}
\otimes z_j
\end{equation*}
and fix $g_j,$ $z_j$ given by that Lemma. Now, let $\varepsilon>0$ with $\frac{1}{\varepsilon}\in\mathbb{N}$. As we have already mentioned, the value of $\e$ will be fixed in Section \ref{sec:proofofall}. We divide $\tor$ into  $(\frac{1}{\varepsilon})^d$ cubes of side length $\varepsilon$ and denote by $\mathcal{Q}$ the collection of these cubes. For each $Q\in\mathcal{Q}$ and $j\in\lbrace 1,\dots, 2d\rbrace,$ we define 
\begin{equation}\label{eq:meanvaluecubes}
g_j^Q =\frac{1}{|Q|}\int_{Q}g_j(x)\, dx.
\end{equation}
Note that we have, by Lemma \ref{lemma:decomposition},
\begin{align}
|g_j^Q|&\leq\|R_0\|_{L^\infty(\tor)},\label{eq:gjqinlinf}\\
|g_j^Q|&\leq \frac{1}{\varepsilon^d} \|R_0\|_{L^1(Q)}\text{ and}\label{eq:gjqbetrag}\\
\sum_{Q\in\mathcal{Q}}\|\mathbb{1}_Qg_j^Q\|_{L^1(Q)}& = \|g_j\|_{L^1(\tor)}\leq \|R_0\|_{L^1(\tor)},\label{eq:gjqsumme}
\end{align}
where $\mathbb{1}_Q$ is the characteristic function of the cube $Q.$  We also fix the constant $$S:=\min\left\lbrace |g_j^Q|: 1\leq j\leq 2d, Q\in\mathcal{Q} \text{ with } g_j^Q\neq 0\right\rbrace.$$

\subsubsection{Choice of Mikado, for each cube $Q \in \mathcal{Q}$}
\label{ss:choice-mikado}

Let $\mu_0>0$ be such that
\begin{align}\label{eq:growthconditionagain}
D|\tau|\leq\beta(\tau)\leq C|\tau| \text{ for all } |\tau|\geq S^\frac{1}{p}\mu_0^\frac{d-1}{p}
\end{align} 
(recall \eqref{grwothcondition}). 
Now let $\zeta>0$ and $\mu \geq \mu_0$. As we have already mentioned, the value of $\zeta$ will be fixed in Section \ref{sec:proofofall}. 

We now define for each pair $(Q,j)\in \mathcal{Q}\times\lbrace 1,\dots,2d\rbrace$ a Mikado Density and a Mikado Field $\Theta^{Q,j} \in C^\infty(\tor)$ and  Mikado Field $W^{Q,j} \in C^\infty(\tor;\mathbb{R}^d)$ (depending on $\zeta,\mu$) as follows. Let $(Q,j)\in \mathcal{Q}\times\lbrace 1,\dots,2d\rbrace$ and let $k\in\lbrace 1,\dots, d\rbrace$ such that $z_j = e_k.$ In case of $g_j^Q=0$, we define
\begin{align*}
\Theta^{Q,j} = 0,\\
W^{Q,j} = 0.
\end{align*}
Otherwise, we set $a=g_j^Q,$ $\sigma=(-1)^{j+1}$ and apply Proposition \ref{prop:mikadosgeneral} with $a,$ $k,$ $\zeta,$ $\mu$ and $\sigma$ to get
\begin{align*}
\Theta^{Q,j} := \Theta_{a,k,\zeta,\mu,\sigma},\\
W^{Q,j} := W_{a,k,\zeta,\mu,\sigma}.
\end{align*}
By an abuse of notation, we translate the functions $\Theta^{Q,j}, W^{Q,j}$ such that $\supp(\Theta^{Q,j})=\supp(W^{Q,j})$ and $\supp(\Theta^{Q,j})\cap \supp(W^{Q,i}) = \emptyset$ for $i\neq j.$

\smallskip

We collect some properties of the building blocks for later reference, which are just \emph{copy and paste} from Proposition \ref{prop:mikadosgeneral}.

\begin{prop}\label{prop:defofmikados}
The following holds.
\begin{equation}\label{eq:propertieofmikados}
\begin{cases}
\dv W^{Q,j} &= 0,\\
\dv \Theta^{Q,j}W^{Q,j} &= 0,\\
\int_{\mathbb{T}^{d}}W^{Q,j}\, dx & = 0,\\
\left|\int_{\mathbb{T}^{d}}\beta(|g_j^Q|^\frac{1}{p}\Theta^{Q,j})W^{Q,j}\, dx - (-1)^{j+1}|g_j^Q|^\frac{1}{p}z_j\right| &<\zeta
\end{cases}
\end{equation}
and in case of $g_j^Q\neq 0$ also
\begin{equation}\label{eq:meanvalueproperty}
\left|\int_{\mathbb{T}^{d}}\Theta^{Q,j}W^{Q,j}\, dx - z_j \right|< \zeta.
\end{equation}
It holds for all $r\in [1,\infty]$
\begin{align}
\|\Theta^{Q,j}\|_{L^r(\tor)}&\leq M_0\mu^{\frac{d-1}{p}-\frac{d-1}{r}},\label{eq:thetainlr}\\
\|W^{Q,j}\|_{L^r(\tor)}&\leq M_0\mu^{\frac{d-1}{p'}-\frac{d-1}{r}}\text{ and}\label{eq:winlr}\\
\|\Theta^{Q,j}\W\|_{L^r(\tor)}&\leq M_0^2\mu^{d-1-\frac{d-1}{r}}
\end{align}
where $M_0$ is the constant in the statement of Proposition \ref{prop:mikadosgeneral}. 
Furthermore, it holds for all $m\in\mathbb{N}$ and $r\in\lbrack 1,\infty\rbrack$
\begin{align*}
\|D^m\Theta^{Q,j}\|_{L^r(\mathbb{T}^d)}&\leq C(\varepsilon,\zeta,\|R_0\|_{L^1(\tor)},m)\mu^{\frac{d-1}{p}+m-\frac{d-1}{r}},\\
\|D^mW^{Q,j}\|_{L^r(\mathbb{T}^d)}&\leq C(\varepsilon,\zeta,\|R_0\|_{L^1(\tor)},m) \mu^{\frac{d-1}{p'}+m-\frac{d-1}{r}}.
\end{align*}
In particular, we have that
\begin{align}
\|\Theta^{Q,j}\|_{L^p(\tor)} &\leq M_0,\label{eq:thetainlp}\\
\|W^{Q,j}\|_{L^{p'}(\tor)}&\leq M_0,\label{eq:winlp'}\\
\|\Theta^{Q,j}\W\|_{L^1(\tor)}&\leq M_0^2\text{ and}\label{eq:thetawproduct}\\
\|DW^{Q,j}\|_{L^{\tilde{p}}(\tor)}&\leq C(\varepsilon,\zeta,\|R_0\|_{L^1(\tor)})\mu^{-\gamma}\label{eq:dwlptilde}
\end{align}
with $\gamma = \frac{d-1}{\tilde{p}}-\frac{d-1}{p'}-1>0$ because of \eqref{eq:d-1}.
Also,  we have that $\operatorname{supp}(\Theta^{Q,j})=\operatorname{supp}(W^{Q,j})$ and $\Theta_{\mu}^{Q,j}, W_\mu^{Q,i}$ only have non-disjoint support if $i=j.$
\end{prop}

\subsubsection{Choice of cutoffs, for each cube $Q \in \mathcal{Q}$}
\label{ss:cutoffs}

Let now $\alpha \in (0,1)$. As we have already mentioned, its  value will be fixed in Section \ref{sec:proofofall}. For each cube $Q\in\mathcal{Q},$ we define cutoffs $\psi_{Q}, \chi_Q\in C^\infty(\tor)$ with the following properties. For $x\in Q,$ we prescribe
\begin{align*}
\chi_Q(x)=1 \text{ if } \operatorname{dist}(x,\partial Q)>\alpha\varepsilon \text{ and } \chi_Q(x)=0 \text{ if } \operatorname{dist}(x,\partial Q)<\frac{\alpha}{2}\varepsilon,\\
\psi_Q(x)=1 \text{ if } \operatorname{dist}(x,\partial Q)>\frac{\alpha}{2}\varepsilon \text{ and }\psi_Q(x)=0 \text{ if } \operatorname{dist}(x,\partial Q)<\frac{\alpha}{4}\varepsilon,
\end{align*}
and $\chi_Q(x)=\psi_Q(x) =0$ for $x\notin Q.$
 Note that this implies
\begin{align}
\nabla \psi_Q&\neq 0 \Rightarrow \chi_Q=0 \text{ and}\nonumber\\
\nabla \chi_Q&\neq 0 \Rightarrow \psi_Q=1.\label{cutoffs}
\end{align}
We will oftentimes use that
\begin{align}
\sum_{Q\in\mathcal{Q}}\|\psi_Q\|_{C^m(\tor)},\sum_{Q\in\mathcal{Q}}\|\chi_Q\|_{C^m(\tor)},\sum_{Q\in\mathcal{Q}} \|\psi_Q\chi_Q\|_{C^m(\tor)}&\leq C(\varepsilon,\alpha, m).\label{eq:cutoffest}
\end{align} 

\subsubsection{Definition of perturbations $\nu,w$ and correctors $\nu_c, w_c$}
\label{ss:def-perturbations}

Let now $\lambda \in \frac{1}{\e} \N$. As we have already mentioned, its value will be fixed in Section \ref{sec:proofofall}. We set
\begin{align}
\label{eq:perturbations-1}
\rho_1 = \rho_0 + \nu + \nu_c,\hspace{0,3cm} u_1 = u_0 +w + w_c
\end{align}
where $\nu,\nu_c,w,w_c$ are defined as follows. With the Mikado Densities and Fields from \mbox{Proposition \ref{prop:defofmikados},} we set
\begin{equation}
\label{eq:perturbations-2}
\begin{aligned}
\nu(x) &=\sum_{j=1}^{2d}\sum_{Q\in\mathcal{Q}}\psi_Q(x)|g_j^Q|^\frac{1}{p}\Theta^{Q,j}(\lambda x), \\
\nu_c &= - \int_{\tor}\nu(x)\, dx,\\
w(x)&=\sum_{j=1}^{2d}\sum_{Q\in\mathcal{Q}}\chi_Q(x)\sign(g_j^Q)|g_j^Q|^\frac{1}{p'}W^{Q,j}(\lambda x).
\end{aligned}
\end{equation}
We will sometimes abbreviate 
\begin{align*}
\Theta^{Q,j}(\lambda \, \cdot) = \left(\Theta^{Q,j}\right)_\lambda,\hspace{0,5cm}W^{Q,j}(\lambda \, \cdot) = \left(W^{Q,j}\right)_\lambda.
\end{align*}
In order to define the corrector $w_c,$ we first calculate
$$-\dv w = -\jQsum \sign(g_j^Q)|g_j^Q|^\frac{1}{p'}\nabla\chi_Q\cdot\left(W^{Q,j}\right)_\lambda.$$
Since $\W$ has zero mean value, we can apply Lemma \ref{lemma:antidiv} and define 
\begin{equation}
\label{eq:perturbations-3}
w_c = -\jQsum  \sign(g_j^Q)|g_j^Q|^\frac{1}{p'}\mathcal{R}\left(\nabla\chi_Q\cdot\left(W^{Q,j}\right)_\lambda\right).
\end{equation}
This gives $\dv(w+w_c)=0$ and thus $\dv(u_1) = \dv(u_0) + \dv(w+w_c)=0.$

\subsection{Estimates for the perturbations}
\label{ss:estimates-perturbation}
We estimate now the size of the perturbations in the relevant norms. Recall that $M_0$ is the constant from Proposition \ref{prop:mikadosgeneral} and Proposition \ref{prop:defofmikados}. 
\begin{lemma}\label{lemma:nulp}
It holds
\begin{equation*}
\|\nu\|^p_{L^p(\tor)}\leq 2dM_0^p\|R_0\|_{L^1(\tor)}+\frac{C(\varepsilon,\alpha,\|R_0\|_{L^1(\tor)})}{\lambda}.
\end{equation*}
\end{lemma}
\begin{proof}
First note that we have
\begin{align*}
\|\nu\|_{L^p(\tor)}^p &=\int_{\tor}\jQsum |\psi_Q(x)|g_j^Q|^\frac{1}{p}|^p|\Theta^{Q,j}(\lambda x)|^p\, dx\\
&= \sum_{j=1}^{2d}\int_{\tor}\left|\sum_{Q\in\mathcal{Q}}\psi_Q(x)|g_j^Q|^\frac{1}{p}\right|^p \left|\sum_{Q\in\mathcal{Q}}\mathbb{1}_Q(x)\Theta^{Q,j}(\lambda x)\right|^p \, dx = \sum_{j=1}^{2d}\int_{\tor} |f^j(x)|^p|h^j(x)|^p\, dx 
\end{align*}
with
\begin{align*}
f^j(x) &:=\sum_{Q\in\mathcal{Q}}\psi_Q(x)|g_j^Q|^\frac{1}{p},\\
h^j(x) &:=\sum_{Q\in\mathcal{Q}}\mathbb{1}_Q(x)\Theta^{Q,j}(\lambda x).
\end{align*}
Fix $j\in\lbrace 1,\dots, 2d\rbrace.$ Let us divide $\tor$ into $\frac{1}{\lambda}$ cubes of edge $\frac{1}{\lambda}.$ Since, by assumption, $\lambda \geq \frac{1}{\varepsilon},$  for each such cube $\tilde{Q}_0$ with edge length $\frac{1}{\lambda}$ there is a "super cube" $Q_0\in\mathcal{Q}$ of edge $\varepsilon$ such that $\tilde{Q}_0\subset Q_0.$ On each cube $\tilde{Q}_0$ of edge $\frac{1}{\lambda}$ with super cube $Q_0\in\mathcal{Q},$ we have
\begin{align}\label{eq:hoeldergeneralised}
\int_{\tilde{Q}_0}|f^j(x)|^p|h^j(x)|^p \,dx&=\int_{\tilde{Q}_0}\left[|f^j(x)|^p - \left(\frac{1}{|\tilde{Q}_0|}\int_{\tilde{Q}_0}|f^j(y)|^p\, dy\right)\right] |h^j(x)|^p\, dx\nonumber\\
&\hspace{0,3cm} +\frac{1}{|\tilde{Q}_0|}\int_{\tilde{Q}_0}|f^j(y)|^p\, dy\int_{\tilde{Q}_0} |h^j(x)|^p \,dx.
\end{align}
We take the first summand and sum over all cubes $\tilde{Q}_0$ of edge $\frac{1}{\lambda}$ and use that for all $x,y$ in the same cube $\tilde{Q}_0$ it holds $\left||f(x)|^p-|f(y)|^p\right|\leq \frac{C_p}{\lambda}\|f\|_{C^1(\tor)}^p.$ We obtain
\begin{align*}
&\sum_{\tilde{Q}_0}\int_{\tilde{Q}_0}\left[|f^j(x)|^p - \left(\frac{1}{|\tilde{Q}_0|}\int_{\tilde{Q}_0}|f^j(y)|^p\, dy\right)\right] |h^j(x)|^p\, dx\\
&\hspace{0,3cm}\leq \frac{C_p}{\lambda} \|f^j\|^p_{C^1(\tor)}\sum_{\tilde{Q}_0}\int_{\tilde{Q}_0}|h^j(x)|^p\, dx\\
&\hspace{0,3cm}= \frac{C_p}{\lambda}\|f^j\|^p_{C^1(\tor)} \|h^j\|^p_{L^p(\tor)}.
\end{align*} 
For the second term in \eqref{eq:hoeldergeneralised}, we estimate using \eqref{eq:thetainlp}
\begin{align*}
&\frac{1}{|\tilde{Q}_0|}\int_{\tilde{Q}_0}|f^j(y)|^p\, dy\int_{\tilde{Q}_0}|h^j(x)|^p\, dx\\
&\hspace{0,3cm}=\frac{1}{|\tilde{Q}_0|}\int_{\tilde{Q}_0}|f^j(y)|^p\, dy \int_{\tilde{Q}_0}\mathbb{1}_{Q_0}(x)|\Theta^{Q_0,j}(\lambda x)|^p\, dx\\
&\hspace{0,3cm}=\frac{1}{|\tilde{Q}_0|}\int_{\tilde{Q}_0}|f^j(y)|^p\, dy\,\frac{1}{\lambda^d}\int_{\tor}|\Theta^{Q_0,j}(\lambda x)|^p\, dx\\
&\hspace{0,3cm}=\|f^j\|^p_{L^p(\tilde{Q}_0)} \|\Theta^{Q_0,j}\|_{L^p(\tor)}^p\\
&\hspace{0,3cm}\leq M_0^p\|f^j\|_{L^p(\tilde{Q}_0)}^p.
\end{align*}
Summing over all cubes $\tilde{Q}_0$ of edge $\frac{1}{\lambda},$ we obtain
\begin{align*}
\sum_{\tilde{Q}_0}\frac{1}{|\tilde{Q}_0|}\int_{\tilde{Q}_0}|f^j(y)|^p\, dy\int_{\tilde{Q}_0} |h^j(x)|^p\, dx\leq M_0^p\|f^j\|_{L^p(\tor)}^p.
\end{align*}
Together, we have
\begin{align*}
\sum_{j=1}^{2d}\int_{\tor}|f^j(x)|^p|h^j(x)|^p\, dx \leq\sum_{j=1}^{2d}\left(\frac{C_p}{\lambda}\|f^j\|^p_{C^1(\tor)}\|h^j\|^p_{L^p(\tor)} + M_0^p\|f^j\|^p_{L^p(\tor)}\right).
\end{align*}
Finally, we note that by \eqref{eq:gjqbetrag} and \eqref{eq:cutoffest}
\begin{align*}
\|f^j\|^p_{C^1(\tor)}&\leq \left(\sum_{Q\in\mathcal{Q}}|g_j^Q|^\frac{1}{p}\|\psi_Q\|_{C^1(\tor)}\right)^p \leq\frac{1}{\varepsilon^d}\|R_0\|_{L^1(\tor)}\left(\sum_{Q\in\mathcal{Q}}\|\psi_Q\|_{C^1(\tor)}\right)^p\\
& = C(\varepsilon, \alpha, \|R_0\|_{L^1(\tor)}),
\end{align*}
and \eqref{eq:thetainlp} gives
\begin{align*}
\|h^j\|^p_{L^p(\tor)}& =\sum_{Q\in\mathcal{Q}}\|\left(\Theta^{Q,j}\right)_\lambda\|^p_{L^p(Q)} = \sum_{Q\in\mathcal{Q}}\varepsilon^d \|\Theta^{Q,j}\|^p_{L^p(\tor)}\leq M_0^p
\end{align*}
and because of \eqref{eq:gjqsumme} and $|\psi_Q|\leq 1$ we have
\begin{align*}
\|f^j\|^p_{L^p(\tor)} = \sum_{Q\in\mathcal{Q}}\int_Q|\psi_Q|g_j^Q|^\frac{1}{p}|^p\, dx \leq \sum_{Q\in\mathcal{Q}}\int_Q|g_j^Q|\, dx = \sum_{Q\in\mathcal{Q}}\|\mathbb{1}_Qg_j^Q\|_{L^1(Q)}\leq \|R_0\|_{L^1(\tor)}.
\end{align*}
 This proves the assertion.
\end{proof}

\begin{rem}
We essentially followed the argument used in the proof of the so called \emph{improved H\"older inequality} (see e.g. \cite[Lemma 2.1]{modena2018non}). 
\end{rem}

\begin{lemma}\label{lemma:nucl1}
We have 
$$|\nu_c|\leq \|\nu\|_{L^1(\tor)}\leq C(\varepsilon,\|R_0\|_{L^1(\tor)})\mu^{-\frac{d-1}{p'}}$$
and
$$\|w\|_{L^1(\tor)}\leq C(\varepsilon,\|R_0\|_{L^1(\tor)})\mu^{-\frac{d-1}{p}}.$$
\end{lemma}

\begin{proof}
It holds by definition of $\nu_c$ that $|\nu_c|\leq\|\nu\|_{L^1(\tor)}.$ Using $\supp(\psi_Q)\subset Q,$ $|\psi_Q|\leq 1,$ \eqref{eq:gjqbetrag} and \eqref{eq:thetainlr}, we get
\begin{align*}
|\nu_c|\leq\|\nu\|_{L^1(\tor)}&\leq \jQsum\|\psi_Q|g_j^Q|^\frac{1}{p}(\Theta^{Q,j})_\lambda\|_{L^1(\tor)}\\
&\leq C(\varepsilon,\|R_0\|_{L^1(\tor)})\sum_{j=1}^{2d}\sum_{Q\in\mathcal{Q}}\|(\Theta^{Q,j})_\lambda\|_{L^1(Q)}\\
&=C(\varepsilon,\|R_0\|_{L^1(\tor)})\sum_{j=1}^{2d}\sum_{Q\in\mathcal{Q}}\varepsilon^d\|\Theta^{Q,j}\|_{L^1(\tor)}\\
&\leq C(\varepsilon,\|R_0\|_{L^1(\tor)})2dM_0\mu^{-\frac{d-1}{p'}}\\
&= C(\varepsilon,\|R_0\|_{L^1(\tor)})\mu^{-\frac{d-1}{p'}}.
\end{align*}
A similar computation with \eqref{eq:winlr} instead of \eqref{eq:thetainlr} shows the estimate for $\|w\|_{L^1(\tor)}$.
\end{proof}

\begin{lemma}\label{lemma:winlp'}
It holds
\begin{equation*}
\|w\|^{p'}_{L^{p'}(\tor)}\leq 2dM_0^p\|R_0\|_{L^1(\tor)}+\frac{C(\varepsilon,\alpha,\|R_0\|_{L^1(\tor)})}{\lambda}.
\end{equation*}
\begin{proof}
The proof is completely analogous to the proof of Lemma \ref{lemma:nulp} with $\|W^{Q,j}\|_{L^{p'}(\tor)}$ instead of   $\|\Theta^{Q,j}\|_{L^{p}(\tor)}.$
\end{proof}
\end{lemma}

\begin{lemma}\label{lemma:winwptilde}
We have the estimate
$$\|w\|_{W^{1,\tilde{p}}(\tor)}\leq C(\varepsilon,\alpha,\zeta,\|R_0\|_{L^1(\tor)})\lambda\mu^{-\gamma}.$$
\end{lemma}
\begin{proof}
We calculate
\begin{align*}
Dw(x)=\jQsum \sign(g_j^Q)|g_j^Q|^\frac{1}{p'}\W(\lambda x)\otimes D\chi_Q +\lambda\chi_Q \sign(g_j^Q)|g_j^Q|^\frac{1}{p'}D\W(\lambda x).
\end{align*}
We take the norm in $L^{\tilde{p}}$ and use \eqref{eq:gjqbetrag}, \eqref{eq:dwlptilde} and \eqref{eq:cutoffest} to get
\begin{align*}
\|Dw\|_{L^{\tilde{p}}(\tor)}&\leq \jQsum |g_j^Q|^\frac{1}{p'}\|\chi_Q\|_{C^1(\tor)}\left(\|\W\|_{L^{\tilde{p}}(\tor)}  + \lambda \|D\W\|_{L^{\tilde{p}}(\tor)}\right)\\
&\leq C(\varepsilon,\alpha,\zeta,\|R_0\|_{L^1(\tor)})\lambda\mu^{-\gamma}.
\end{align*}
\end{proof}

\begin{lemma}\label{lemma:wcinlp'}
It holds
\begin{align*}
\|w_c\|_{L^{p'}(\tor)}\leq\frac{C(\varepsilon,\alpha, \|R_0\|_{L^1(\tor)})}{\lambda}.
\end{align*}
\end{lemma}

\begin{proof}
With Lemma \ref{lemma:antidiv} and the inequalities  \eqref{eq:gjqbetrag},\eqref{eq:winlp'} and \eqref{eq:cutoffest}, we obtain
\begin{align*}
\|w_c\|_{L^{p'}(\tor)}&\leq\jQsum\frac{C_{0,p'}|g_j^Q|^\frac{1}{p'}}{\lambda}\|\nabla\chi_Q\|_{C^2(\tor)}\|\W\|_{L^{p'}(\tor)}\\
&\leq \frac{C(\varepsilon,\alpha, \|R_0\|_{L^1(\tor)})}{\lambda}.
\end{align*}
\end{proof}

\begin{lemma}\label{lemma:wcinwptilde}
It holds
\begin{align*}
\|w_c\|_{W^{1,\tilde{p}}(\tor)}\leq C(\varepsilon,\alpha, \zeta, \|R_0\|_{L^1(\tor)})\mu^{-\gamma}.
\end{align*}
\end{lemma}

\begin{proof}
We use Lemma \ref{lemma:antidiv} again together with \eqref{eq:gjqbetrag}, \eqref{eq:dwlptilde} and \eqref{eq:cutoffest} and get
\begin{align*}
\|Dw_c\|_{L^{\tilde{p}}(\tor)}&\leq\jQsum |g_j^Q|^\frac{1}{p'}\|\nabla\chi_Q\|_{C^3(\tor)}\|\W\|_{W^{1,\tilde{p}}(\tor)}\\
&\leq C(\varepsilon,\alpha,\zeta, \|R_0\|_{L^1(\tor)})\mu^{-\gamma}.
\end{align*}
\end{proof}

\section{The Defect field}
\label{sec:defect}

We now define the new defect field $R_1 = R_1(x) \in \R^{2 \times d}$. We first define (in Section \ref{sec:firsteqdef}) and estimate (in Section \ref{ss:estimates-r1}) the first row $R_1^1$ of $R_1$ and then we define and estimate (in Section \ref{sec:secondeq}) the second row of $R_1$. 

\subsection{Definition of the first row of the new Defect Field}\label{sec:firsteqdef}
So far we have only defined the new density $\rho_1$ and the new vector field $u_1$. Let us now focus on the new error $R_1.$ We want to define the new error tensor such that the first row $R_1^1$ satisfies
\begin{equation*}
-\dv R_1^1 = \dv(\rho_1u_1).
\end{equation*}
We calculate, using $\dv u_1 = 0$ and $\dv (\rho_0u_0) = -R_0^1:$
\begin{align*}
\dv(\rho_1u_1)&= \dv(\nu w+\rho_0u_0)+\dv(\nu u_0 + \rho_0 w)+ \dv(\nu_c u_1)+ \dv(\rho_0w_c + \nu w_c)\\
&=\dv(\nu w-R_0^1) + \dv(\nu u_0+ \rho_0 w) +  \dv(\rho_0 w_c + \nu w_c)\\
&= \dv(\nu w - R_0^1+ R^{1,\operatorname{lin}} + R^{1,\operatorname{corr}})
\end{align*}
with
\begin{align*}
R^{1,\operatorname{lin}} & = \nu u_0+\rho_0 w ,\\
R^{1,\operatorname{corr}} & = \rho_0 w_c + \nu w_c.
\end{align*}
Furthermore, we have
\begin{align*}
\nu w - R_0^1 &= \jQsum\psi_Q\chi_Q g_j^Q\left(\Theta^{Q,j} W^{Q,j}\right)_\lambda-\sum_{j=1}^{2d} g_jz_j\\
&=\jQsum \psi_Q\chi_Q g_j^Q\left\lbrack\left(\Theta^{Q,j}W^{Q,j}\right)_\lambda-\int_{\tor}\Theta^{Q,j}W^{Q,j}\, dx\right\rbrack\\
&\hspace{0,3cm}+\jQsum \psi_Q\chi_Qg_j^Q\left\lbrack \int_{\tor}\Theta^{Q,j}W^{Q,j}\, dx - z_j\right\rbrack\\ 
&\hspace{0,3cm} +\jQsum \psi_Q\chi_Q g_j^Qz_j-\sum_{j=1}^{2d} g_jz_j\\
&= \jQsum \psi_Q\chi_Q g_j^Q\left\lbrack\left(\Theta^{Q,j}W^{Q,j}\right)_\lambda-\int_{\tor}\Theta^{Q,j}W^{Q,j}\, dx\right\rbrack +R^{1,\operatorname{mean}} + R^{\chi,\psi}
\end{align*}
with 
\begin{align*}
R^{1,\operatorname{mean}} &= \jQsum \psi_Q\chi_Qg_j^Q\left\lbrack \int_{\tor}\Theta^{Q,j}W^{Q,j}\, dx - z_j\right\rbrack,\\
R^{\chi,\psi} &= \jQsum \psi_Q\chi_Q g_j^Qz_j-\sum_{j=1}^{2d} g_jz_j.
\end{align*}
Taking the divergence of the remaining term in the last line of the previous computation, we see that 
\begin{align*}\operatorname{div}\left(\jQsum\psi_Q\chi_Qg_j^Q\left\lbrack\left(\Theta^{Q,j}W^{Q,j}\right)_\lambda-\int_{\tor}\Theta^{Q,j}W^{Q,j}\, dx\right\rbrack\right)\\
=\jQsum g_j^Q\nabla(\psi_Q\chi_Q)\cdot\left\lbrack\left(\Theta^{Q,j}W^{Q,j}\right)_\lambda-\int_{\tor}\Theta^{Q,j}W^{Q,j}\, dx\right\rbrack,
\end{align*}
and we define $R^{1,\operatorname{quad}}$ as
$$R^{1,\operatorname{quad}}=\jQsum  g_j^Q\mathcal{R}\left(\nabla(\psi_Q\chi_Q)\cdot\left\lbrack  
\underbrace{ \left(\Theta^{Q,j}W^{Q,j}\right)_\lambda-\int_{\tor}\Theta^{Q,j}W^{Q,j}\, dx }_{\text{fast oscillating with zero mean value} }   \right \rbrack\right).$$
Let us thus define 
\begin{equation*}
R_1^1 = -\left(R^{1,\operatorname{quad}} + R^{1,\operatorname{mean}} + R^{\chi,\psi} +  R^{1,\operatorname{lin}} + R^{1,\operatorname{corr}}\right),
\end{equation*}
which satisfies $-\dv R^1_1 = \dv(\rho_1 u_1)$ by construction. In the next section we will estimate $R_1^1$ in $L^1.$

\subsection{Estimates of the first row of the new Defect Field}
\label{ss:estimates-r1}
We investigate each term in the definition of $R_1^1$ separately.

\begin{lemma}\label{lemma:1quad}
It holds
\begin{equation*}
\|R^{1,\operatorname{quad}}\|_{L^1(\tor)}\leq \frac{C(\varepsilon,\alpha,\|R_0\|_{L^1(\tor)})}{\lambda}.
\end{equation*}
\end{lemma}

\begin{proof}
We use Lemma \ref{lemma:antidiv}, \eqref{eq:gjqbetrag}, \eqref{eq:thetawproduct} and \eqref{eq:cutoffest} and   estimate 
\begin{align*}
\|R^{1,\operatorname{quad}}\|_{L^1(\tor)}&\leq\jQsum \frac{C_{0,1}|g_j^Q|}{\lambda}\|\nabla(\psi_Q\chi_Q)\|_{C^2(\tor)}\left\|\left(\Theta^{Q,j}W^{Q,j}\right)_\lambda-\int_{\tor}\Theta^{Q,j} W^{Q,j}\, dx\right\|_{L^1(\tor)}\\
&\leq \frac{C(\varepsilon,\alpha, \|R_0\|_{L^1(\tor)})}{\lambda}.
\end{align*}
\end{proof}

\begin{lemma}\label{lemma:r1mean}
We have 
\begin{align*}
\|R^{1,\operatorname{mean}}\|_{L^1(\tor)}\leq 2d \|R_0\|_{L^1(\tor)} \zeta .
\end{align*}
\end{lemma}

\begin{proof}
We estimate each term in the definition of $R^{1,\operatorname{mean}}$ separately. Note that we only need to estimate those terms with $g_j^Q\neq 0.$ For such, we have with $\supp(\psi_Q\chi_Q)\subset Q,$ $|\psi_Q\chi_Q|\leq 1$ and using \eqref{eq:meanvalueproperty}  
\begin{align*}
 \left\|\psi_Q\chi_Qg_j^Q\left\lbrack \int_{\tor}\Theta^{Q,j}W^{Q,j}\, dx - z_j\right\rbrack\right\|_{L^1(\tor)}\leq \zeta \|\mathbb{1}_Qg_j^Q\|_{L^1(Q)}.
\end{align*}
Hence, we obtain with \eqref{eq:gjqsumme}
\begin{align*}
\|R^{1,\operatorname{mean}}\|_{L^1(\tor)}&\leq \zeta\jQsum\|\mathbb{1}_Qg_j^Q\|_{L^1(Q)}\leq  2d \zeta\|R_0\|_{L^1(\tor)}.
\end{align*}

\end{proof}

\begin{lemma}\label{lemma:rpsichi}
It holds
\begin{align*}
\|R^{\chi,\psi}\|_{L^1(\tor)}\leq \sum_{j=1}^{2d}\big\|(\sum_{Q\in\mathcal{Q}}\mathbb{1}_Qg_j^Q)-g_j\big\|_{L^1(\tor)} + 2d\|R_0\|_{L^\infty(\tor)}(1-\alpha^d).
\end{align*}
\end{lemma}

\begin{proof}
We have
\begin{align*}
R^{\chi,\psi} &= \jQsum \psi_Q\chi_Q g_j^Qz_j-\sum_{j=1}^{2d} g_jz_j \\
&= \jQsum \psi_Q\chi_Q g_j^Qz_j - \jQsum \psi_Q\chi_Q g_j z_j\\
 &\hspace{0,2cm} + \jQsum \psi_Q\chi_Q g_j z_j - \sum_{j=1}^{2d} g_jz_j .
\end{align*}
We estimate the first line in the previous computation, using $|\psi_Q\chi_Q|\leq 1,$ and $\supp(\psi_Q\chi_Q)\subset Q$
\begin{align*}
&\left\|\jQsum \psi_Q\chi_Q g_j^Qz_j - \jQsum \psi_Q\chi_Q g_j z_j\right\|_{L^1(\tor)}\leq \jQsum\|g_j^Q - g_j\|_{L^1(Q)}\\
&\hspace{1cm} = \sum_{j=1}^{2d}\big\|(\sum_{Q\in\mathcal{Q}}\mathbb{1}_Qg_j^Q)-g_j\big\|_{L^1(\tor)}.
\end{align*}
For the second line, we define
\begin{align*}
D = \bigcup\limits_{Q\in\mathcal{Q}}\lbrace x\in Q: \psi_Q\chi_Q\neq 1\rbrace\subset\tor
\end{align*}
and estimate, using $|g_j(x)|\leq \|R\|_{L^{\infty}(\tor)}$
\begin{align*}
&\left\| \jQsum \psi_Q\chi_Q g_j z_j - \sum_{j=1}^{2d}g_jz_j\right\|_{L^1(\tor)}\leq \sum_{j=1}^{2d}\big\|\big\lbrack (\sum_{Q\in\mathcal{Q}}\psi_Q\chi_Q) -1 \big\rbrack g_j\big\|_{L^1(\tor)}\\
&\hspace{1cm} \leq\sum_{j=1}^{2d} \operatorname{vol}(D)\|g_j\|_{L^\infty(\tor)} = 2d\|R_0\|_{L^\infty(\tor)}\operatorname{vol}(D) \\
&\hspace{1cm} \leq 2d\|R_0\|_{L^\infty(\tor)}(1-\alpha^d).
\end{align*}
\end{proof}

\begin{lemma}\label{lemma:r1lin}
It holds 
\begin{align*}
\|R^{1,\operatorname{lin}}\|_{L^1(\tor)}\leq C(\varepsilon,\|R_0\|_{L^1(\tor)},\|u_0\|_{C^0(\tor)},\|\rho_0\|_{C^0(\tor)})\mu^{-\min(\frac{d-1}{p},\frac{d-1}{p'})}.
\end{align*}
\end{lemma}

\begin{proof}
We have $\|R^{1,\operatorname{lin}}\|_{L^1(\tor)}\leq \|u_0\|_{C^0(\tor)}\|\nu\|_{L^1(\tor)} + \|\rho_0\|_{C^0(\tor)}\|w\|_{L^1(\tor)}.$
The first summand can be estimated using Lemma \ref{lemma:nucl1} by
\begin{align*}
\|u_0\|_{C^0(\tor)}\|\nu\|_{L^1(\tor)} &\leq C(\varepsilon, \|R_0\|_{L^1(\tor)})\|u_0\|_{C^0(\tor)}\mu^{-\frac{d-1}{p'}} = C(\varepsilon, \|R_0\|_{L^1(\tor)},\|u_0\|_{C^0(\tor)})\mu^{-\frac{d-1}{p'}}.
\end{align*}
Similarly, the second summand can be estimated by $$\|\rho_0\|_{C^0(\tor)}\|w\|_{L^1(\tor)}\leq C(\varepsilon,\|R_0\|_{L^1(\tor)},\|\rho_0\|_{C^0(\tor)})\mu^{-\frac{d-1}{p}}.$$
\end{proof}

\begin{lemma}\label{lemma:r1corr}
It holds
\begin{align*}
\|R^{1,\operatorname{corr}}\|_{L^1(\tor)}\leq\frac{C(\varepsilon,\alpha, \|R_0\|_{L^1(\tor)},\|\rho_0\|_{C^0(\tor)})}{\lambda}.
\end{align*}
\end{lemma}
\begin{proof}
We have $\|R^{1,\operatorname{corr}}\|_{L^1(\tor)}\leq \|\rho_0\|_{C^0(\tor)}\|w_c\|_{L^1(\tor)} + \|\nu\|_{L^p(\tor)}\|w_c\|_{L^{p'}(\tor)}.$
With Lemma \ref{lemma:nulp} and \ref{lemma:wcinlp'}, we obtain
\begin{align*}
\|R^{1,\operatorname{corr}}\|_{L^1(\tor)}&\leq \|\rho_0\|_{C^0(\tor)}\|w_c\|_{L^{p'}(\tor)} + \|\nu\|_{L^p(\tor)}\|w_c\|_{L^{p'}(\tor)}\\
&\leq\frac{C(\varepsilon,\alpha, \|R_0\|_{L^1(\tor)},\|\rho_0\|_{C^0(\tor)})}{\lambda}.
\end{align*}

\end{proof}

\subsection{Definition and Estimates of the second row of the new Defect Field}\label{sec:secondeq}

We do a similar computation as in the beginning of the previous section. We have (recalling that $\div (\beta(\rho_0) u_0 - h_0) = - \div R^2_0$)
\begin{align*}
\dv(\beta(\rho_{1})u_{1}-h^\ast)&=\dv(\beta(\rho_0+\nu + \nu_c)u_0)+\dv(\beta(\rho_0+\nu + \nu_c)w)+\dv(\beta(\rho_0+\nu + \nu_c)w_c)-\dv h^\ast\\
&=\dv(\beta(\rho_0)u_0 - h_0)-\dv(\beta(\rho_0)u_0)+\dv(\beta(\rho_0+\nu+\nu_c)u_0)\\
&\hspace{0,5cm} +\dv(\beta(\rho_0+\nu+\nu_c)w)+\dv(\beta+\nu+\nu_c)w_c) + \dv(h_0-h^\ast)\\
&=\dv (\beta(\nu)w -R_0^2) \\
&\hspace{0,5cm} + \dv(\lbrack \beta(\rho_0+\nu+\nu_c)-\beta(\rho_0)\rbrack u_0) + \dv([\beta(\rho_0+\nu+\nu_c)-\beta(\nu)]w)\\
&\hspace{0,5cm}  +\dv(\beta(\rho_0+\nu+\nu_c)w_c)\\
&\hspace{0,5cm} + \dv(h_0-h^\ast)\\
& = \dv (\beta(\nu)w - R_0^2) + \dv R^{2,\operatorname{lin}} + \dv R^{2,\operatorname{corr}} + \dv R^{h}
\end{align*}
with
\begin{align*}
R^{2,\operatorname{lin}} &= \lbrack \beta(\rho_0+\nu+\nu_c)-\beta(\rho_0)\rbrack u_0 + [\beta(\rho_0+\nu+\nu_c)-\beta(\nu)]w,\\
R^{2,\operatorname{corr}} &=\beta(\rho_0+\nu+\nu_c)w_c,\\
R^h &= h_0 - h^\ast.
\end{align*}

We sketch the estimates for $R^{2,\operatorname{lin}}, R^{2\operatorname{corr}}$ and $R^h$.

\begin{lemma}\label{lemma:r2lin}
It holds
\begin{align*}
\|R^{2,\operatorname{lin}}\|_{L^1(\tor)}\leq C(\varepsilon, \|R_0\|_{L^1(\tor)},\|\rho_0\|_{C^0(\tor)}, \|u_0\|_{C^0(\tor)})\mu^{-\min(\frac{d-1}{p},\frac{d-1}{p'})}.
\end{align*}
\end{lemma}

\begin{proof}
We have by the Lipschitz continuity of $\beta$
\begin{align*}
\|\lbrack \beta(\rho_0+\nu+\nu_c)-\beta(\rho_0)\rbrack u_0\|_{L^1(\tor)}\leq \operatorname{Lip}(\beta)\|u_0\|_{C^0(\tor)} \|\nu+\nu_c\|_{L^1(\tor)}
\end{align*}
and
\begin{align*}
\|[\beta(\rho_0+\nu+\nu_c)-\beta(\nu)]w\|_{L^1(\tor)}\leq \operatorname{Lip}(\beta) (\|\rho_0\|_{C^0(\tor)}\|w\|_{L^1(\tor)} + |\nu_c|\|w\|_{L^1(\tor)}).
\end{align*}
Now Lemma \ref{lemma:nucl1} implies the statement.
\end{proof}

\begin{lemma}\label{lemma:r2corr}
It holds
\begin{align*}
\|R^{2,\operatorname{corr}}\|_{L^1(\tor)}\leq\frac{C(\varepsilon,\alpha, \|R_0\|_{L^1(\tor)},\|\rho_0\|_{C^0(\tor)})}{\lambda}.
\end{align*}
\end{lemma}
\begin{proof}
Using again the Lipschitz continuity of $\beta$ and $\beta(0) = 0$, we see that
\begin{align*}
\|R^{2,\operatorname{corr}}\|_{L^1(\tor)}\leq \operatorname{Lip}(\beta)\| |\rho_0+\nu+\nu_c|w_c\|_{L^1(\tor)}
\end{align*}
and this can be estimated just as $R^{1,\operatorname{corr}}.$
\end{proof}

\begin{lemma}\label{lemma:rh}
It holds $\|R^h\|_{L^1(\tor)}\leq \frac{\delta}{4}.$
\end{lemma}
\begin{proof}
This is true by the assumptions on $h_0, h^\ast.$
\end{proof}

Now we come to the remaining terms. Using that $\beta(0)=0$, $\supp \psi_Q \subseteq Q$ and, for any fixed $Q$, $\Theta^{Q,i}$, $\Theta^{Q,j}$ have mutually disjoint support for $i \neq j$, we have
$$\beta(\nu) = \jQsum\beta(\psi_Q|g_j^Q|^\frac{1}{p}\left(\Theta^{Q,j}\right)_\lambda)$$
and we calculate 
\begin{align*}
\beta(\nu)w-R_0^2 &=\jQsum \beta\left(\psi_Q|g_j^Q|^\frac{1}{p}\left(\Theta^{Q,j}\right)_\lambda\right)\chi_Q\sign(g_j^Q)|g_j^Q|^\frac{1}{p'}\left(W^{Q,j}\right)_\lambda-\sum_{j=1}^{2d}(-1)^{j+1}g_jz_j\\
&=\jQsum \chi_Q\sign(g_j^Q)|g_j^Q|^\frac{1}{p'}\left[\beta\left(\psi_Q|g_j^Q|^\frac{1}{p}\left(\Theta^{Q,j}\right)_\lambda\right)\left(W^{Q,j}\right)_\lambda\right.\\
&\hspace{0,3cm}\left.-\int_{\tor}\beta\left(|g_j^Q|^\frac{1}{p}\left(\Theta^{Q,j}\right)_\lambda\right)\left(W^{Q,j}\right)_\lambda\, dx\right\rbrack\\
&\hspace{0,3cm} + \jQsum \chi_Q\sign(g_j^Q) |g_j^Q|^\frac{1}{p'}\left[\int_{\tor}\beta\left(|g_j^Q|^\frac{1}{p}\left(\Theta^{Q,j}\right)_\lambda\right)\left(W^{Q,j}\right)_\lambda\, dx - (-1)^{j+1}|g_j^Q|^\frac{1}{p}z_j \right]\\
&\hspace{0,3cm} + \jQsum (-1)^{j+1}\chi_Qg_j^Qz_j-\sum_{j=1}^{2d}(-1)^{j+1}g_jz_j\\
&= \jQsum \chi_Q\sign(g_j^Q)|g_j^Q|^\frac{1}{p'}\left\lbrack\beta\left(\psi_Q|g_j^Q|^\frac{1}{p}\left(\Theta^{Q,j}\right)_\lambda\right)\left(W^{Q,j}\right)_\lambda\right.\\
&\hspace{0,3cm}-\left.\int_{\tor}\beta\left(|g_j^Q|^\frac{1}{p}\left(\Theta_{\mu}^{Q,j}\right)_\lambda\right)\left(W_{\mu}^{Q,j}\right)_\lambda\, dx\right\rbrack + R^{2,\operatorname{mean}} + R^\chi
\end{align*}
with 
\begin{align*}
 R^{2,\operatorname{mean}} & = \jQsum \chi_Q \sign(g_j^Q)|g_j^Q|^\frac{1}{p'}\left[\int_{\tor}\beta\left(|g_j^Q|^\frac{1}{p}\left(\Theta^{Q,j}\right)_\lambda\right)\left(W^{Q,j}\right)_\lambda\, dx - (-1)^{j+1} |g_j^Q|^\frac{1}{p}z_j\right],\\
 R^\chi & = \jQsum (-1)^{j+1}\chi_Qg_j^Qz_j-\sum_{j=1}^{2d}(-1)^{j+1}g_jz_j.
\end{align*}

\begin{lemma}\label{lemma:r2mean}
We have $$\|R^{2,\operatorname{mean}}\|_{L^1(\tor)}\leq 2d \|R_0\|_{L^\infty(\tor)}^\frac{1}{p'}  \zeta.$$ 
\end{lemma}
\begin{proof}
We have with $\supp(\chi_Q)\subset Q,$ $|\psi_Q\chi_Q|\leq 1$ and using the last line of \eqref{eq:propertieofmikados}  
\begin{align*}
 \left\|\chi_Q|g_j^Q|^\frac{1}{p'}\left\lbrack \int_{\tor}\beta\left(|g_j^Q|^\frac{1}{p}\left(\Theta^{Q,j}\right)_\lambda\right)\left(W^{Q,j}\right)_\lambda\, dx - (-1)^{j+1}|g_j^Q|^\frac{1}{p}z_j\right\rbrack\right\|_{L^1(\tor)}\leq \zeta \|\mathbb{1}_Q|g_j^Q|^\frac{1}{p'}\|_{L^1(Q)}.
\end{align*}
Hence, we obtain with \eqref{eq:gjqinlinf}
\begin{align*}
\|R^{2,\operatorname{mean}}\|_{L^1(\tor)}&\leq \zeta\jQsum\|\mathbb{1}_Q|g_j^Q|^\frac{1}{p'}\|_{L^1(Q)} = \zeta\jQsum |g_j^Q|^\frac{1}{p'}\varepsilon^d\\
&\leq  2d \zeta\|R_0\|_{L^\infty(\tor)}^\frac{1}{p'}.
\end{align*}
\end{proof}

\begin{lemma}\label{lemma:r2chi}
We have
$$\|R^\chi\|_{L^1(\tor)}\leq \sum_{j=1}^{2d}\big\|(\sum_{Q\in\mathcal{Q}}\mathbb{1}_Qg_j^Q)-g_j\big\|_{L^1(\tor)} + 2d\|R_0\|_{L^\infty(\tor)}(1-\alpha^d) .$$
\end{lemma}
\begin{proof}
This can be estimated exactly as $R^{\chi,\psi}$.
\end{proof}

Finally we deal with the remaining term in the above calculation for $\beta(\nu)w - R_0^2$. Recall that by \eqref{cutoffs} the gradient $\nabla\chi_Q$ is only nonzero if $\psi_Q=1,$ while $\nabla\psi_Q$ is only nonzero if $\chi_Q=0.$ Using $\operatorname{div}(W^{Q,j}) = \operatorname{div}(\Theta^{Q,j} W^{Q,j}) = 0,$ we have
\begin{align*}
&\operatorname{div}\left(\jQsum \chi_Q\sign(g_j^Q)|g_j^Q|^\frac{1}{p'}\left\lbrack\beta\left(\psi_Q|g_j^Q|^\frac{1}{p}\left(\Theta^{Q,j}\right)_\lambda\right)\left(W^{Q,j}\right)_\lambda\right.\right.\\
&\hspace{0,3cm}\left.\left.-\int_{\tor} \beta(|g_j^Q|^\frac{1}{p}\left(\Theta^{Q,j}\right)_\lambda)\left(W^{Q,j}\right)_\lambda\, dx\right\rbrack\right)\\
&=\jQsum \sign(g_j^Q)|g_j^Q|^\frac{1}{p'}\nabla\chi_Q\cdot\left\lbrack\beta\left(\underbrace{\psi_Q}_{=1}|g_j^Q|^\frac{1}{p}\left(\Theta^{Q,j}\right)_\lambda\right)\left(W^{Q,j}\right)_\lambda\right.\\
&\hspace{0,3cm}\left.-\int_{\tor} \beta\left(|g_j^Q|^\frac{1}{p}\left(\Theta^{Q,j}\right)_\lambda\right)\left(W^{Q,j}\right)_\lambda\, dx\right\rbrack\\
&\hspace{0,2cm} + \jQsum \underbrace{\chi_Q}_{=0}\sign(g_j^Q)|g_j^Q|^\frac{1}{p'}\left\lbrack\beta'\left(\psi_Q|g_j^Q|^\frac{1}{p}\left(\Theta^{Q,j}\right)_\lambda\right)|g_j^Q|^\frac{1}{p}\left(\Theta^{Q,j}\right)_\lambda\nabla\psi_Q\cdot\left(W^{Q,j}\right)_\lambda\right\rbrack\\
&=\jQsum \sign(g_j^Q)|g_j^Q|^\frac{1}{p'}\nabla\chi_Q\cdot\left\lbrack\beta\left(|g_j^Q|^\frac{1}{p}\left(\Theta^{Q,j}\right)_\lambda\right)\left(W^{Q,j}\right)_\lambda\right.
\\&\hspace{0,3cm}\left.-\int_{\tor} \beta(|g_j^Q|^\frac{1}{p}\left(\Theta^{Q,j}\right)_\lambda)\left(W^{Q,j}\right)_\lambda\, dx\right\rbrack.
\end{align*}
Hence, we can define
\begin{align*}
R^{2,\operatorname{quad}}=\jQsum\sign(g_j^Q)|g_j^Q|^\frac{1}{p'}\mathcal{R}\left(\nabla\chi_Q\cdot \left\lbrack\beta\left(|g_j^Q|^\frac{1}{p}\left(\Theta^{Q,j}\right)_\lambda\right)\left(W^{Q,j}\right)_\lambda\right.\right.\\
\left.\left.-\int_{\tor} \beta\left(|g_j^Q|^\frac{1}{p}\left(\Theta^{Q,j}\right)_\lambda\right)\left(W^{Q,j}\right)_\lambda\, dx\right\rbrack\right).
\end{align*} 
Notice that the definition is well posed, in the sense that the operator $\mathcal{R}$ is applied, for each fixed $(Q,j)$ in the summation, to the product $\nabla \chi_Q (F^{Q,j})_\lambda$ of the ``slow oscillating'' function $\nabla \chi_Q$ and the ``fast oscillating'' function (with zero mean value) $(F^{Q,j})_\lambda$ where 
\begin{equation*}
F^{Q,j} : \T^d \to \R^d, \qquad F^{Q,j}(x) = \beta \Big(|g_j^Q|^{ \frac{1}{p}  } \Theta^{Q,j}(x) \Big) W^{Q,j}(x) - \int_{\T^d} \beta \Big(|g_j^Q|^{1/p} \Theta^{Q,j}(x) \Big) W^{Q,j}(x)  dx.
\end{equation*}

With that definition, we set
\begin{align*}
R^2_1 =-\left( R^{2,\operatorname{quad}} + R^{2,\operatorname{mean}} + R^{\chi} +  R^{2,\operatorname{lin}} + R^{2,\operatorname{corr}} + R^h\right),
\end{align*}
which satisfies
\begin{equation*}
-\dv R^2_1 = \dv(\beta(\rho_1)u_1-h^*).
\end{equation*}
\begin{lemma}\label{lemma:r2quad}
It holds
\begin{equation*}
\|R^{2,\operatorname{quad}}\|_{L^1(\tor)}\leq \frac{C(\varepsilon,\alpha,\|R_0\|_{L^1(\tor)})}{\lambda}.
\end{equation*}
\end{lemma}
\begin{proof}
For the map $F^{Q,j}$ introduced above, we have, using the Lipschitz continuity of $\beta$, 
\begin{equation*}
\begin{aligned}
\|F^{Q,j}\|_{L^1(\T^d)} 
& \leq 2 \left\| \beta \Big(|g_j^Q|^{ \frac{1}{p} } \Theta^{Q,j}(x) \Big) W^{Q,j}(x) \right\|_{L^1(\T^d)} \\
& \leq 2 \Lip(\beta) |g_j^Q|^{\frac{1}{p}} \|\Theta^{Q,j} W^{Q,j}\|_{L^1(\T^d)}  \\
\text{by \eqref{eq:thetawproduct} } 
& \leq  2 \Lip(\beta) M_0^2.
\end{aligned}
\end{equation*}
As for $R^{1,\operatorname{quad}}$, we then have
\begin{equation*}
\begin{aligned}
\|R^{2, \operatorname{quad}}\|_{L^1(\T^d)}
& \leq \frac{1}{\lambda} \jQsum |g_j^Q|^{\frac{1}{p'}} \|\chi_Q\|_{C^2} \left\| F^{Q,j} \right\|_{L^1(\T^d)} \leq \frac{C(\varepsilon,\alpha,\|R_0\|_{L^1(\tor)})}{\lambda}.
\end{aligned}
\end{equation*}
\end{proof}

\section{Proof of the Proposition}\label{sec:proofofall}
In this section we conclude the proof of Proposition \ref{prop:main} by fixing the parameters introduced at the beginning of Section \ref{sec:perturbations} and showing the estimates \eqref{est:rho} -- \eqref{est:R}. Note that \eqref{eq:equalitiesofR1} is satisfied by construction.
We start with estimate \eqref{est:R} and recall that by definition of $R_1^1:$ 
\begin{align}\label{eq:r1inl1bysums}
\|R_1^1\|_{L^1(\tor)}& \leq \|R^{1,\operatorname{quad}}\|_{L^1(\tor)} + \|R^{1,\operatorname{mean}}\|_{L^1(\tor)} + \|R^{\chi,\psi}\|_{L^1(\tor)} +  \|R^{1,\operatorname{lin}}\|_{L^1(\tor)} + \|R^{1,\operatorname{corr}}\|_{L^1(\tor)}.
\end{align}
We choose $\varepsilon$ so small such  that for each $j\in \lbrace 1,\dots, 2d\rbrace,$ the mean values $(g_j^Q)_{Q\in\mathcal{Q}}\subset\mathbb{R}$ defined by $g_j^Q =\frac{1}{|Q|}\int_{Q}g_j(x) \, dx$ satisfy
\begin{align*}
\big\|g_j-\sum_{Q\in\mathcal{Q}}\mathbb{1}_Qg_j^Q\big\|_{L^1(\tor)}<\frac{\delta}{16d}.
\end{align*}
Furthermore, we let $\alpha\in (0,1)$ be such that 
\begin{align*}
\|R_0\|_{L^\infty(\tor)}(1-\alpha^d)\leq\frac{\delta}{16d}.
\end{align*}
With $\varepsilon$ and $\alpha$ fixed, the estimate from Lemma \ref{lemma:rpsichi} becomes
\begin{align*}
\|R^{\chi,\psi}\|_{L^1(\tor)}&\leq \sum_{j=1}^{2d}\big\|(\sum_{Q\in\mathcal{Q}}\mathbb{1}_Qg_j^Q)-g_j\big\|_{L^1(\tor)} + 2d\|R_0\|_{L^\infty(\tor)}(1-\alpha^d)\\
&\leq\frac{\delta}{8} + \frac{\delta}{8} = \frac{\delta}{4},
\end{align*}
and for the same reason we also have (for the term $R^\chi$ in the definition of the second row $R^2_1$ of the new error)
$$\|R^\chi\|_{L^1(\tor)}\leq \frac{\delta}{4}.$$
Next, we choose $\zeta>0$ such that 
$$\zeta\|R_0\|_{L^1(\tor)}\leq \frac{\delta}{8d} \text{ and } \zeta\|R_0\|_{L^\infty(\tor)}^\frac{1}{p'}\leq \frac{\delta}{8d}. $$
With that choice, the estimates from Lemma \ref{lemma:r1mean} and \ref{lemma:r2mean}  become
\begin{align*}
\|R^{1,\operatorname{mean}}\|_{L^1(\tor)}\leq 2d \zeta\|R_0\|_{L^1(\tor)}&\leq \frac{\delta}{4},\\
\|R^{2,\operatorname{mean}}\|_{L^1(\tor)}\leq 2d \zeta\|R_0\|_{L^\infty(\tor)}^\frac{1}{p'}&\leq \frac{\delta}{4}.
\end{align*}
Now, putting together Lemmas \ref{lemma:1quad}, \ref{lemma:r1mean}, \ref{lemma:rpsichi}, \ref{lemma:r1lin} and \ref{lemma:r1corr}, we obtain from \eqref{eq:r1inl1bysums} by our choice of $\varepsilon, \alpha$ and $\zeta$ that

\begin{align*}
\|R_1^1\|_{L^1(\tor)}& \leq \|R^{1,\operatorname{quad}}\|_{L^1(\tor)} + \|R^{1,\operatorname{mean}}\|_{L^1(\tor)} + \|R^{\chi,\psi}\|_{L^1(\tor)} +  \|R^{1,\operatorname{lin}}\|_{L^1(\tor)} + \|R^{1,\operatorname{corr}}\|_{L^1(\tor)}\\
&\leq \frac{\delta}{2} + C(\varepsilon,\alpha, \|R_0\|_{L^1(\tor)}, \|\rho_0\|_{C^0(\tor)},\|u_0\|_{C^0(\tor)})\left(\frac{1}{\lambda} + \mu^{-\min(\frac{d-1}{p},\frac{d-1}{p'})}\right).
\end{align*}
Similarly, we use Lemma \ref{lemma:r2quad}, \ref{lemma:r2mean}, \ref{lemma:r2chi}, \ref{lemma:r2lin}, \ref{lemma:r2corr} and \ref{lemma:rh} and obtain

\begin{align*}
\|R^2_1\|_{L^1(\tor)} &\leq \|R^{2,\operatorname{quad}}\|_{L^1(\tor)} + \|R^{2,\operatorname{mean}}\|_{L^1(\tor)} + \|R^{\chi}\|_{L^1(\tor)} +  \|R^{2,\operatorname{lin}}\|_{L^1(\tor)}\\
&\hspace{0,3cm} + \|R^{2,\operatorname{corr}}\|_{L^1(\tor)} + \|R^h\|_{L^1(\tor)}\\
&\leq\frac{3\delta}{4} + C(\varepsilon,\alpha, \|R_0\|_{L^1(\tor)}, \|\rho_0\|_{C^0(\tor)},\|u_0\|_{C^0(\tor)})\left(\frac{1}{\lambda} + \mu^{-\min(\frac{d-1}{p},\frac{d-1}{p'})}\right).
\end{align*}
We set
$$\mu = \lambda^c$$
and show that $\lambda$ and $c>1$ can be chosen large enough such that the estimates \eqref{est:rho} -- \eqref{est:R} are satisfied. Since $\varepsilon$ and $\alpha$ have already been fixed and do not depend on $\lambda,$ we can choose $\lambda$ large such that 
\begin{equation*}
\|R^1_1\|_{L^1(\tor)}, \|R^2_1\|_{L^1(\tor)}\leq\delta,
\end{equation*}
which shows estimate \eqref{est:R}.
For \eqref{est:rho}, we use Lemma \ref{lemma:nulp} and \ref{lemma:nucl1} and choose $\lambda$ large enough such that
\begin{align*}
\|\rho_1 - \rho_0\|_{L^p(\tor)}\leq \|\nu\|_{L^p(\tor)} + |\nu_c|\leq (3d)^\frac{1}{p} M_0 \|R_0\|_{L^1(\tor)}^\frac{1}{p}.
\end{align*}
Analogously, we use for \eqref{est:ulp} Lemma \ref{lemma:winlp'} and \ref{lemma:wcinlp'} and obtain for $\lambda$ large enough
\begin{align*}
\|u_1 - u_0\|_{L^{p'}(\tor)}\leq (3d)^\frac{1}{p'} M_0 \|R_0\|_{L^1(\tor)}^\frac{1}{p'}.
\end{align*}
Hence, \eqref{est:rho} and \eqref{est:ulp} are satisfied with $M$ as defined in \eqref{eq:constant-m}. For estimate \eqref{est:uw}, we use Lemma \ref{lemma:winwptilde} and Lemma \ref{lemma:wcinwptilde} and get
\begin{align*}
\|u_1 - u_0\|_{W^{1,\tilde{p}}(\tor)}\leq C(\varepsilon,\alpha,\zeta,\|R_0\|_{L^1(\tor)})\lambda\mu^{-\gamma}\leq \delta
\end{align*}
if we fix $c>\frac{1}{\gamma}$ and choose $\lambda$ large enough.

\section{The Hamiltonian Case}
\label{sec:hamiltonian}

This last section is devoted to give a sketch of the proof of Theorem \ref{thm:hamiltonian}. Let $d = 2d' \geq 4$ be even. We can modify our building blocks such that the vector field perturbations in each iteration step and also the limit function $u$ are Hamiltonian, i.e. $u = J\nabla H$ for some function $H:\tor\rightarrow\mathbb{R}\in W^{1,p'}(\tor)\cap W^{2,\tilde{p}}(\tor)$ with
\begin{equation*}
J = \begin{pmatrix}
0_{d'} & I_{d'}\\
-I_{d'} & 0_{d'}
\end{pmatrix}.
\end{equation*}
To achieve this, we modify the proof of Lemma \ref{lemma:construction1} and Proposition \ref{prop:defofmikados}. 

\begin{prop}
Let $\zeta>0,$ $\varepsilon>0$ with $\frac{1}{\varepsilon}\in\mathbb{N}$ and let $\mathcal{Q}$ be a partition of $\tor$ into cubes of side length $\varepsilon.$ Let $(g^Q_j)_{Q\in\mathcal{Q}, j\in\lbrace 1,\dots, 2d\rbrace}$ be defined as in \eqref{eq:meanvaluecubes} and let $\mu_0$ be defined by \eqref{eq:growthconditionagain}. Let $\mu \geq \mu_0$.  For every pair $(Q,j)\in\mathcal{Q}\times\lbrace 1,\dots, 2d\rbrace$ we have a Mikado Density
$\Theta^{Q,j}\in C^\infty(\tor)$ and  a Mikado Field  $W^{Q,j}\in C^\infty(\tor;\mathbb{R}^d)$ satisfying the following properties:
\begin{equation}\label{eq:hamiltonianproperties}
\begin{cases}
\dv W^{Q,j} &= 0,\\
\dv \Theta^{Q,j}W^{Q,j} &= 0,\\
\int_{\mathbb{T}^{d}}W^{Q,j}\, dx & = 0,\\
\left|\int_{\mathbb{T}^{d}}\beta(|g_j^Q|^\frac{1}{p}\Theta^{Q,j})W^{Q,j}\, dx -(-1)^{j+1} |g_j^Q|^\frac{1}{p}z_j\right| &<\zeta
\end{cases}
\end{equation}
and in case of $g_j^Q\neq 0$ also
\begin{equation}\label{eq:meanvaluepropertyhamiltonian}
\left|\int_{\mathbb{T}^{d}}\Theta^{Q,j}W^{Q,j}\, dx - z_j \right|< \zeta.
\end{equation}
We have the additional property 
$$W^{Q,j} = J\nabla H^{Q,j}$$ for some function $H^{Q,j}\in C^\infty(\tor).$ The estimates \eqref{eq:thetainlr} -- \eqref{eq:dwlptilde} are valid for some constant $M_2$ independent of $Q,j,\mu,\zeta$ and $\varepsilon$, together with the additional estimate for $H^{Q,j}$:
\begin{equation}\label{est:hamiltonian}
\|H^{Q,j}\|_{L^r(\tor)}\leq M_2\mu^{\frac{d-1}{p'}-1+m-\frac{d-1}{r}}.
\end{equation}
Also,  we have that $\operatorname{supp}(\Theta^{Q,j})=\operatorname{supp}(W^{Q,j})$ and $\Theta^{Q,j}, W^{Q,i}$ only have non-disjoint support if $i=j.$
\end{prop}
\begin{proof}
Let $(Q,j)\in\mathcal{Q}\times\lbrace 1,\dots, 2d\rbrace$ and $\mu\geq\mu_0$ be fixed. Similar to the proof of Lemma \ref{lemma:construction1}, let $P_1=(\frac{3}{4},\dots,\frac{3}{4}), P_2 = (\frac{1}{4},\frac{3}{4},\dots,\frac{3}{4}), P_3 =(\frac{3}{4},\frac{1}{4},\frac{3}{4},\dots,\frac{3}{4})\in\mathbb{R}^{d-1}.$ We first define $\varphi,\tilde{\varphi}:\mathbb{R}^{d-1}\rightarrow\mathbb{R}$ by
\begin{equation*}
\varphi = \chi_{B_{\frac{1}{16}}(P_1)} - \chi_{B_{\frac{1}{16}}(P_2)}
\end{equation*}
and
\begin{equation}
\tilde{\varphi} = \alpha_1\chi_{B_{\frac{1}{8}}(P_1)}+\alpha_2\chi_{B_{\frac{1}{8}}(P_2)} + \alpha_3 \chi_{B_{\frac{1}{8}}(P_3)}
\end{equation}
with $\alpha=(\alpha_1, \alpha_2,\alpha_3)\in\mathbb{R}^3$  which can be chosen (depending on $g_j^Q$ and $\mu$) such that the associated concentrated functions
\begin{align*}
\psi(x) := \mu^{\frac{d-1}{p}}\varphi(\mu x) \text{ and}\\
\tilde{\psi}(x) := \mu^{\frac{d-1}{p'}}\tilde{\varphi}(\mu x)
\end{align*}
satisfy
\begin{align}
\int_{\mathbb{R}^{d-1}}\psi\tilde \psi\, dx &= 1,\\
\int_{\mathbb{R}^{d-1}}\beta(|g_j^Q|^\frac{1}{p}\psi)\tilde \psi\, dx &=(-1)^{j+1} |g_j^Q|^\frac{1}{p}\text{ and}\\
\int_{\mathbb{R}^{d-1}}\tilde \psi \, dx&= 0.
\end{align}
One can show just as in Lemma \ref{lemma:construction1} that there is a constant $\tilde{M}$ independent of $Q,j,\mu,\zeta$ and $\varepsilon$ such that $|\alpha|\leq \tilde{M}$ and therefore also $|\varphi,\tilde{\varphi}|\leq \tilde{M}.$ Again we  set $\Phi := \eta_{\ell}\ast \varphi$ with $\ell = \ell(\varepsilon,\zeta, \|R_0\|_{L^1(\tor)})$ so small such that
\begin{align*}
\|\Phi-\varphi\|_{L^p(\mathbb{R}^{d-1})}&<  \frac{\zeta}{2\tilde{M}}\text{ and}\\
|g_j^Q|^\frac{1}{p'}\|\Phi-\varphi\|_{L^p(\mathbb{R}^{d-1})}&\leq \frac{\|R_0\|_{L^1(\tor)}}{\varepsilon^d}\|\Phi-\varphi\|_{L^p(\mathbb{R}^{d-1})}<\frac{\zeta}{2\tilde{M}\operatorname{Lip}(\beta)}
\end{align*}
and such that $$\supp(\Phi)\subseteq B_{\frac{1}{8}}(P_1)\cup B_{\frac{1}{8}}(P_2).$$  We denote again by $\Psi$ the concentrated function $$\Psi(x) := \mu^{\frac{d-1}{p}}\Phi(\mu x).$$
Let $\tilde{\Phi}\in C^\infty(\mathbb{R}^{d-1})$ be a function with 
\begin{align*}
\supp(\tilde{\Phi})\subseteq B_{\frac{1}{4}}(P_1)\cup B_{\frac{1}{4}}(P_2) \cup B_{\frac{1}{4}}(P_3), \hspace{0,3cm} \int_{\mathbb{R}^{d-1}}\tilde{\Phi} \, dx = 0
\end{align*}
and
\begin{align*}
 \tilde{\Phi}(x) = \alpha_1x_1  \text{ on } B_{\frac{1}{8}}(P_1), \tilde{\Phi}(x) = \alpha_2x_1 \text{ on } B_{\frac{1}{8}}(P_2)\text{ and } \tilde{\Phi}(x) = \alpha_3x_1 \text{ on } B_{\frac{1}{8}}(P_3). 
\end{align*}
We can also assume that $|\tilde{\Phi}|\leq \tilde{M}$ and $\|\nabla\tilde{\Phi}\|\leq 2\tilde{M}.$
Let $\tilde{\Psi}$ be the associated concentrated function given by
$$\tilde{\Psi}= \mu^{\frac{d-1}{p'}-1}\tilde{\Phi}(\mu x).$$ 
Let $k\in \lbrace 1,\dots, d\rbrace$ such that $z_j=e_k$ and  $\bar{j}\in\lbrace 1,\dots, d\rbrace$ such that $J e_{\bar{j}} = \pm z_j.$  We can now define $\Theta^{Q,j}, H^{Q,j}$ as the 1-periodic extension of the functions
\begin{align*}
\tilde{\Theta}^{Q,j}(x) &= \Psi(x_{\bar{j}}, x_1, \dots,x_{k-1},x_{k+1},\dots,x_{\bar{j}-1}, x_{\bar{j}+1},\dots,x_d),\\
\tilde{H}^{Q,j}(x) &= \pm\tilde{\Psi}(x_{\bar{j}}, x_1, \dots,x_{k-1},x_{k+1},\dots,x_{\bar{j}-1}, x_{\bar{j}+1},\dots,x_d).
\end{align*}
We define $W^{Q,j}$ accordingly as 
$$W^{Q,j} = J\nabla H^{Q,j}.$$
 For $x=(x_1,\dots, x_d)$ we abbreviate $\bar{x} = (x_{\bar{j}}, x_1, \dots,x_{k-1},x_{k+1},\dots,x_{\bar{j}-1}, x_{\bar{j}+1},\dots,x_d).$ Notice that for $x\in[0,1]^d$ with $\bar{x}\in B_{\frac{1}{8\mu}}(P_i/\mu),$ $i=1,2,3,$ we have
 \begin{align*}
 H^{Q,j}(x) = \pm\mu^{\frac{d-1}{p'}-1}\tilde{\Phi}(\mu \bar{x}) =\pm \mu^{\frac{d-1}{p}}\alpha_ix_{\bar{j}}
\end{align*}  
 and therefore also
\begin{align*}
 W^{Q,j}(x) = J\nabla H^{Q,j}(x) = \tilde \psi(\bar{x})z_j \text{ on } \supp(\Theta^{Q,j}).
\end{align*}
With this definition, the first three properties in \eqref{eq:hamiltonianproperties} are easily verified. The estimates \eqref{eq:thetainlr} -- \eqref{eq:dwlptilde} together with the estimate \eqref{est:hamiltonian} for $H^{Q,j}$ are shown as in Lemma \ref{lemma:construction1} with $M_2 = 2\tilde{M}$ because of $|\Phi,\tilde{\Phi}|,\|\nabla\tilde{\Phi}\|_{L^\infty(\mathbb{R}^{d-1})}\leq 2\tilde{M}.$ We check \eqref{eq:meanvaluepropertyhamiltonian}. If $g_j^Q\neq 0,$ we have
\begin{align*}
\left|\int_{\tor}\Theta^{Q,j} W^{Q,j}\, dx  - z_j\right| &= \left|\int_{\tor}\Theta^{Q,j} W^{Q,j} - \psi(\bar{x})\tilde \psi(\bar{x}) z_j\, dx\right|\\
 &=\left| \int_{\tor}\Psi(\bar{x})\tilde \psi(\bar{x})z_j - \psi(\bar{x})\tilde \psi(\bar{x}) z_j\, dx\right|\\
 &\leq\|\Psi-\psi\|_{L^p(\mathbb{R}^{d-1})}\|\tilde \psi\|_{L^p(\mathbb{R}^{d-1})}\\
 & = \|\Phi-\varphi\|_{L^p(\mathbb{R}^{d-1})}\|\tilde \psi\|_{L^p(\mathbb{R}^{d-1})}\\
 &\leq \tilde{M}\|\tilde \psi\|_{L^p(\mathbb{R}^{d-1})} <\zeta.
\end{align*}
The fourth property in \eqref{eq:hamiltonianproperties} is shown similarly. Finally, as before, we translate $\Theta^{Q,j}, W^{Q,j}$ in order to get  $\Theta^{Q,j}, W^{Q,j}$ such that  $\Theta^{j,k}, W^{Q,i}$ only have non-disjoint support if $j=i.$
\end{proof}
With those Mikado Densities and Mikado flows, the perturbations $\nu$ and $w$ are defined as in Section \ref{sec:perturbations}.   Additionally, we redefine the corrector $w_c$ by setting
\begin{align*}
w_c =\jQsum \sign(g_j^Q)|g_j^Q|^\frac{1}{p'} (J\nabla\chi_Q)\frac{\left(H^{Q,j}\right)_\lambda}{\lambda}.
\end{align*}
With this definition, we have
\begin{align*}
w+w_c = \jQsum\frac{\sign(g_j^Q)|g_j^Q|^\frac{1}{p'}}{\lambda} J\nabla\left(\chi_Q\left(H^{Q,j}\right)_\lambda\right),
\end{align*}
i.e., the vector field perturbations are Hamiltonian. With  \eqref{eq:gjqbetrag}, \eqref{eq:thetainlr}, \eqref{eq:cutoffest} and  \eqref{est:hamiltonian}, one easily sees that
\begin{align*}
\|w_c\|_{L^{p'}(\tor)}&\leq\jQsum \frac{|g_j^Q|^\frac{1}{p'}}{\lambda}\|J\nabla\chi_Q\|_{C^0(\tor)}\|H^{Q,j}\|_{L^{p'}(\tor)} \\
&\leq C(\varepsilon,\alpha,\|R_0\|_{L^1(\tor)})(\lambda\mu)^{-1}
\end{align*}
and
\begin{align*}
\|Dw_c\|_{L^{\tilde{p}}(\tor)}&\leq \jQsum |g_j^Q|^\frac{1}{p'}\left[\frac{1}{\lambda}\|J\nabla\chi_Q\|_{C^1(\tor)}\|H^{Q,j}\|_{L^{\tilde{p}}(\tor)} + \|J\nabla\chi_Q\|_{C^0(\tor)} \underbrace{\|DH^{Q,j}\|_{L^{\tilde{p}}(\tor)}}_{\approx\|W^{Q,j}\|_{L^{\tilde{p}}(\tor)}}\right]\\
&\leq C(\varepsilon,\alpha,\|R_0\|_{L^1(\tor)})\mu^{-(\gamma+1)},
\end{align*}
i.e. $$\|w_c\|_{W^{1,\tilde{p}}(\tor)}\leq C(\varepsilon,\alpha,\|R_0\|_{L^1(\tor)})\mu^{-(\gamma+1)}.$$ Therefore, Proposition \ref{prop:main} can be proven just as we did in Section \ref{sec:proofofall} and the limit function in the proof of Theorem \ref{thm:main} is Hamiltonian.

\nocite{*}
\bibliographystyle{alpha}
\bibliography{bibliography}

\end{document}